\documentclass[a4paper,11pt]{amsart}

\usepackage[usenames,dvipsnames]{pstricks}
\usepackage{epsfig}
\usepackage{amsfonts}
\usepackage{amsmath}
\usepackage{graphicx}
\usepackage{float}
\usepackage[active]{srcltx}
\usepackage{enumerate}

\newtheorem{fait}{Fact}[section]
\newtheorem{theoreme}[fait]{Theorem}
\newtheorem{lemme}[fait]{Lemma}

\newtheorem{proposition}[fait]{Proposition}

\newcommand{\Pp}{P} 
\newcommand{\Pmil}{P}  
\newcommand{\PW}{P^W} 

\newcommand{\X}{\mathbf{X}}
\newcommand{\ud}{\mathrm{d}}
\newcommand{\tX}{\boldsymbol{\tau}}
\newcommand{\tB}{\tau}
\newcommand{\tla}{\boldsymbol{\sigma}}
\newcommand{\tlam}{\tla(re^{{}v},m^{-}_1)}
\newcommand{\sB}{\sigma}

\newcommand{\R}{\mathbf{R}}
\newcommand{\N}{\mathbf{N}}
\newcommand{\Z}{\mathbf{Z}}
\newcommand{\egloi}{\stackrel{\mathcal{L}}{=}}

\newcommand{\Wb}{\overline{W}}
\newcommand{\cvgloi}{\stackrel{\mathcal{L}}{\longrightarrow}}
\newcommand{\Wu}{\underline{W}}

\newcommand{\Bp}{c^+}

\newcommand{\Cma}{b^+_{v}}
\newcommand{\Cm}{b^+}
\newcommand{\Dma}{a^+_{v}}

\newcommand{\bm}{b^{-}_1}

\newcommand{\am}{a^{-}_1}

\newcommand{\mm}{m^{-}_1}

\newcommand{\apmd}{a_{{v},i}^{-}}

\newcommand{\bpmd}{b_{{v},i}^{-}}

\newcommand{\apd}[1][2]{a_{{v},#1}^{-}}
\newcommand{\amd}[1][1]{a_{{v},#1}^{-}}
\newcommand{\bpd}[1][2]{b_{{v},#1}^{-}}
\newcommand{\bmd}[1][1]{b_{{v},#1}^{-}}

\newcommand{\Bpd}{c^+_{v}}

\newcommand{\Bpg}{\widehat{c}^+_{v}}

\newcommand{\mpmd}{m^{-}_{{v},i}}

\newcommand{\sam}{\tla^{-}_1}
\newcommand{\smd}{\tla_{{v},1}^{-}}
\newcommand{\spd}{\tla_{{v},2}^{-}}
\newcommand{\smg}{\widehat{\tla}_{{v},1}^{-}}
\newcommand{\spg}{\widehat{\tla}_{{v},2}^{-}}

\newcommand{\LX}{\mathbf{L}}
\newcommand{\LB}{L}

\newcommand{\spmd}{\tla_{{v},i}^{-}}

\newcommand{\tm}{\boldsymbol{\tau}(\mm)}
\newcommand{\LXm}{\LX(\sam,x)}
\newcommand{\LXpmd}{\LX(\spmd,x)}
\newcommand{\nm}{n^1}
\newcommand{\genv}{\Gamma_{v}^{1}\cap\Gamma_{v}^{2}\cap\widehat{\Gamma}_{v}^{1}\cap\widehat{\Gamma}_{v}^{2}}
\newcommand{\penv}{\Gamma_{v}^{3}\cap\Gamma_{v}^{2}\cap\widehat{\Gamma}_{v}^{3}\cap\widehat{\Gamma}_{v}^{2}}

\newcommand{\mpd}[1][2]{m_{{v},#1}^{-}}
\newcommand{\mmd}[1][1]{m_{{v},#1}^{-}}

\newcommand{\mmg}[1][1]{\widehat{m}_{{v},#1}^{-}}

\newcommand{\dpd}{d_{{v}}^{+}}
\newcommand{\dmd}[1][1]{d_{{v},#1}^{-}}

\newcommand{\epd}{e_{{v}}^{+}}
\newcommand{\emd}[1][1]{e_{{v},#1}^{-}}

\newcommand{\Upd}{U_{{v}}^{+}}
\newcommand{\Umd}[1][1]{U_{{v}}^{-}}
\newcommand{\Upg}{\widehat{U}^{+}_{{v}}}
\newcommand{\Umg}[1][1]{\widehat{U}_{{v}}^{-}}

\newcommand{\sigp}{\boldsymbol{\sigma}_{{v}}}
\newcommand{\Ipd}[1][{v}]{I^+_{#1}}
\newcommand{\Imd}[1][{v}]{I^-_{#1}}
\newcommand{\Ipg}[1][{v}]{\widehat{I}^+_{#1}}
\newcommand{\Img}[1][{v}]{\widehat{I}^-_{#1}}
\newcommand{\Imdi}[1][{v}]{I^-_{#1}}

\newcommand{\isp}[1][{v}]{j_{#1}}
\newcommand{\Isp}[1][{v}]{J_{#1}}
\newcommand{\shp}{\tla_{{v}}^{\rho}}
\newcommand{\shm}{\tla_{{v}}^{r}}
\newcommand{\A}{\mathcal{A}^1}
\newcommand{\Appm}{\mathcal{A}^i_{v}}
\newcommand{\B}{\mathcal{B}^1}
\newcommand{\Bi}{\mathcal{B}^i_{v}}
\newcommand{\C}{\mathcal{C}}
\newcommand{\D}{\mathcal{D}}
\newcommand{\hW}{\widehat{W}}
\newcommand{\bd}{b_{{v},3}^{-}}

\renewcommand{\b}{b^{-}_0}
\newcommand{\bim}{b^{-}_{{v},i-1}}

\newcommand{\m}{m^+}
\newcommand{\ma}{m^+_{v}}

\newcommand{\bbb}[1][]{b^-_{{v}_{#1},1}}
\newcommand{\mmm}[1][]{m^-_{{v}_{#1},1}}
\newcommand{\mmmg}[1][]{\widehat{m}^-_{{v}_{#1},1}}
\newcommand{\bet}[1][n]{\beta_{#1}}
\newcommand{\gam}[1][n]{\gamma_{#1}}
\newcommand{\mmu}[1][n]{\mu_{#1}}
\newcommand{\LS}{\xi}

\newcommand{\et}[1][n]{\eta_{#1}}

\newcommand{\al}[1][n]{{v}_{#1}}
\renewcommand{\mag}{\widehat{m}^{+}_{v}}

\begin{document}


\title[Asymptotics for the local time in Brownian environment]{Almost sure asymptotics for the local time of a diffusion in Brownian environment}
\author{Roland DIEL}
\address{Roland DIEL\newline
Laboratoire MAPMO - C.N.R.S. UMR 6628\newline
F\'ed\'eration Denis-Poisson, Universit\'e d'Orl\'eans, (Orl\'eans France)}
\date{}

\begin{abstract}
We study here the asymptotic behavior of the maximum local time $\LX^*(t)$ of the diffusion in Brownian environment. Shi \cite{Shi} proved that, surprisingly, the maximum speed of $\LX^*(t)$ is at least $t\log(\log(\log t))$ whereas in the discrete case it is $t$. We show that $t\log(\log(\log t))$ is the proper rate and that for the minimum speed the rate is the same as in the discrete case (see Dembo, Gantert, Peres and Shi \cite{DemGanPerShi}) namely $t/\log(\log(\log t))$. We also prove a localization result: almost surely for large time, the diffusion has spent almost all the time in the neighborhood of four points which only depend on the environment.
\end{abstract}
\maketitle

\section{Introduction}

Let $\left(W(x)\ ,\ x\in\R\right)$ be a two-sided one-dimensional Brownian motion on $\mathbf{R}$ with $W(0)=0$. We call \emph{diffusion process in the environment $W$} a process $\left(\X(t), t \in \R^+\right)$ whose infinitesimal generator given $W$ is
\begin{displaymath}
	\frac{1}{2}e^{W(x)}\frac{d}{dx}\left( e^{-W(x)}\frac{d}{dx} \right).
\end{displaymath}
Notice that if $W$ were differentiable, $(\X(t),t \in \R^+)$ would be solution of the following stochastic differential equation 
\begin{displaymath}
\left\{ \begin{array}{ll}
\ud \X(t)=\ud \beta(t)-\frac{1}{2}W'(\X(t))\ud t,\\
\X(0)=0
\end{array} \right.
\end{displaymath}
in which $\beta$ is a standard one-dimensional Brownian motion independent of $W$. Of course as we choose for $W$ a Brownian motion, the previous equation does not have any rigorous sense but it explains the denomination \emph{environment} for $W$. 

This process was first introduced by Schumacher  \cite{Schumacher} and Brox \cite{Brox}. It is recurrent and sub-diffusive with asymptotic behavior in $(\log t)^2$. Moreover Brox showed in \cite{Brox} that $\X$ has the property to be localized in the neighborhood of a point $m_{\log t}$ only depending on $t$ and $W$. Later, Tanaka \cite{Tanaka,Tanaka3} and Hu \cite{Hu} obtained deeper localization results. The limit law of $m_{\log t}/(\log t)^2$ and therefore of $\X(t)/(\log t)^2$ is independently made explicit by Kesten \cite{Kesten2} and Golosov \cite{Golosov1}.

Kesten and Golosov's aim was actually to determine the limit law of a random walk in random environment, introduced by Solomon \cite{Solomon}, which is often considered as the discrete analogue of Brox's model. Sinai \cite{Sinai} proved that this random walk $(S_n,n \in \N)$, now called Sinai's walk, has the same limit distribution as Brox's. Hu and Shi \cite{HuShi2} got the almost sure rates of convergence of the $\limsup$ and $\liminf$ of $\X$ and $S$. It appears that these rates are the same. 

In the present paper, we are interested in the local time of the diffusion $\X$. This process, denoted $(\LX(t,x),t\geq 0,x\in\R)$, is the density of the occupation measure of $\X$: $\LX$ is the unique a.s. jointly continuous process such that for each Borel set $A$ and for any $t\geq0$,
\begin{equation}
	\nu_t(A):=\int_0^t\mathbf{1}_A(\X_s)\ud s=\int_\R \mathbf{1}_A(x)\LX(t,x)\ud x.\label{tpsocc}
\end{equation}
We shall see later that such a process exists.
The first results on the behavior of $\LX$ can be found in \cite{Shi} and \cite{HuShi1}. In particular Hu and Shi proved in \cite{HuShi1} that for any $x \in \R$, 
\begin{align*} 
\frac{\log(\LX(t,x))}{\log t} \cvgloi U \wedge \hat{U}, \ t \rightarrow + \infty
\end{align*}
 where $U$ and $\hat{U}$ are independent variables uniformly distributed in $(0,1)$ and  $\cvgloi$ is the convergence in law. Note that in the same paper it is also proved that Sinai's walk local time $\LS$ has the same behavior. The process $\LS$ is the time spent by $S$ at $x$ before time $n$  for $n\in\mathbf{N}$ and $x\in\Z$: $\LS(n,x):=\sum_{k=0}^n\mathbf{1}_{S_k=x}$. 
This result shows that the local time in a fixed point can vary a lot. So to study localization, a good quantity to look at is the maximum of the local time  of $\X$, $$\LX^*(t):=\sup_{x\in\R}\LX(t,x).$$
Shi was the first one to be interested in this process; in \cite{Shi} he gave a lower bound on the $\limsup$ behavior. Almost surely,
\begin{align*}
 \limsup_{t\rightarrow\infty}\frac{\LX^*(t)}{t\log_3(t)}\geq\frac{1}{32}
\end{align*}
where for any $i\in\mathbf{N}^*$, $\log_{i+1}=\log\circ\log_i$ and $\log_1=\log$. In the same paper, he computed the similar rate in the discrete case: define for $n\in\N$, $\LS^*(n):=\sup_{x\in\Z}\LS(n,x)$, there is constant $c\in(0,\infty)$ such that, almost surely,
$$
\limsup_{n\rightarrow\infty}\frac{\LS^*(n)}{n}=c.
$$
Thereby he highlighted a different behavior for the discrete model and for the continuous one.
The limit law of $\LX^*(t)$ was then determined in \cite{AndDie}:
\begin{align*} 
\frac{\LX^*(t)}{t} \cvgloi \frac{1}{\int_{-\infty}^{\infty}e^{-\widetilde{W}(x)}\ud x}, \ t \rightarrow + \infty
\end{align*}
where $\widetilde{W}$ has the same law as the environment conditioned to stay positive. 
But unlike the discrete case (see \cite{GanPerShi}), this result does not allow to obtain an upper bound on the almost sure behavior.

Here we prove that the quantity $\log_3(t)$ is the correct renormalization for the $\limsup$ and an analog result for the $\liminf$: 
\begin{theoreme}\label{thcentral}
 Almost surely,
\begin{align*} 
\frac{1}{32}\leq\limsup_{t\rightarrow\infty}\frac{\LX^*(t)}{t\log_3(t)}\leq \frac{e^2}{2}\textrm{ and}\\
\frac{j^2_{0}}{64}\leq\liminf_{t\rightarrow\infty}\frac{\log_3(t)\LX^*(t)}{t}\leq \frac{e^2\pi^2}{4}.
\end{align*}
where $j_{0}$ is the smallest strictly positive root of Bessel function $J_0$.
\end{theoreme}

We can compare these results to the ones in the discrete case, see \cite{Shi} (and \cite{GanPerShi} for the value of the constant) for the $\limsup$ and \cite{DemGanPerShi} for the $\liminf$: there exists two constants $c,c'\in(0,\infty)$ such that almost surely,
\begin{displaymath} 
\begin{array}{ccc}
\displaystyle \limsup_{n\rightarrow\infty}\frac{\LS^*(n)}{n}=c&\textrm{ and }&\displaystyle \liminf_{n\rightarrow\infty}\frac{\log_3(n)\LS^*(n)}{n}=c'.
\end{array}
\end{displaymath}
As Shi noted it, the $\limsup$ in the continuous case and in the discrete case have a different normalization, but we see with Theorem \ref{thcentral} that the normalization is the same for the $\liminf$.

There is the following heuristic interpretation: in the discrete case like in the continuous one, $\LX^*(t)/t$ behaves approximately in the best case like the inverse of the integral of the exponential of the "steepest" environment and in the worst case like the inverse of the integral of the exponential of the "flattest" one. However if in discrete case there is a steepest environment, in the continuous case it is possible to be as steep as we want, and in the two cases, the environment can be as flat as we want. It explains the difference for the $\limsup$. To compute these rates, we need to study carefully both the behavior of the "steepest" and "flattest" environment and the place where the diffusion spent the most of its time. Indeed, as we said, for large $t$ the process $\X$ is in the neighborhood of a point $m_{\log t}$. This event happens with a probability which tends to $1$ when $t$ goes to infinity, however it does not grow fast enough with $t$ to derive almost sure results. So here we relax the localization of the particle to the neighborhood of four points instead of one. Then we obtain the following theorem:
\begin{theoreme}\label{thmloc}
Fix $\epsilon>0$ and $c_0>0$. There are four processes $m^1_t$, $m^2_t$, $m^3_t$ and $m^4_t$ only depending on the environment $W$ such that if we denote
 $$
 I_t:=\bigcup_{i=1}^4[m^i_t-(\log_2 t)^{4+\epsilon},\ m^i_t+(\log_2t)^{4+\epsilon}],
 $$
 then, almost surely,
 $$
 \lim_{t\rightarrow\infty}\frac{(\log t)^{c_0}}{t}\nu_t\left(\R\setminus I_t\right)=0.
 $$
\end{theoreme}
That is to say, for large $t$ the diffusion  has spent much of its time in the neighborhood of at most four points.
%

The rest of the paper is organized  as follows: in Section \ref{esttec}, we present some technical tools used in the article, in Section \ref{estenv}, we show that environment $W$ has ``good'' properties with a large enough probability, Section \ref{briques} is centered on the study of the local time of the process $\X$ at properly chosen random times, then in Section \ref{supLX}, we look at the behavior of these random times to deduce results in deterministic time and finally in Section \ref{demth}, precise almost sure asymptotic result on the environment are given to finish the proofs of Theorems \ref{thcentral} and \ref{thmloc}.

\section{Three useful theorems and some technical estimates}\label{esttec}
We begin with a useful representation of $\X$ we use throughout the article. The martingale representation theorem says that, given the environment $W$, $\X$ is a Brownian motion rescaled in time and space. Precisely (see \cite{Brox}), there is a Brownian motion $B$ started at $0$, independent of $W$ such that if we define the scale function
\begin{equation}
 	\forall x \in \mathbb{R}, S_W(x):=\int_0^xe^{W(y)}\ud y \label{eqS}
\end{equation}
and the random time change
\begin{equation}
\forall t\geq 0, T_{W,B}(t):=\int_0^te^{-2W(S_W^{-1}(B(s)))}\ud s,\label{eqT}
\end{equation}
then 
\begin{equation}\label{eqx}
\X=S_W^{-1}\circ B\circ T_{W,B}^{-1}.
\end{equation}
To simplify notations, we write $S$ and $T$ for respectively $S_W$ and $T_{W,B}$. Using $(\ref{eqx})$, we easily obtain that a continuous function $\LX$ verifies \eqref{tpsocc} only if 
\begin{align}
	\forall x\in\R,\ \forall t\geq0,\ \LX(t,x)=e^{-W(x)}\LB(T^{-1}(t),S(x))\label{eqLx} 
\end{align}
where $\LB$ is the local time process of the Brownian motion $B$. And so the local time of $\X$ is defined correctly. 

We introduce some notations : for $x,r\geq0$, we denote
\begin{align*}
\tX(x):=\inf\{t\geq0\ ;\ \mathbf{X}(t)=x\}&,\ \tla(r,x):=\inf\{t\geq0\ ;\ \LX(t,x)\geq r\},\\
\tB(x):=\inf\{t\geq0\ ;\ B(t)=x\}\ \ &\text{and }\ \sigma(r,x):=\inf\{t\geq0\ ;\ L(t,x)\geq r\}.
\end{align*}
More generally, the quantities related to $\mathbf{X}$ are written in bold font and the ones related to $B$ in normal font. Furthermore, in the rest of the paper, letter $K$ stands for a universal constant whose value can change from one line to another and for any process $M$, $\tau_M(x)$ will denote the hitting time of height $x$ by $M$. We also denote by $\Pp$ the total probability and by $\PW$ the probability given the environment $W$.

As said before, the study of the process $\X$ is reduced by a change in time and space to the study of a Brownian motion. We therefore use in our proofs the following two theorems (see e.g. \cite{RevYor}) that describe the law of the local time of a Brownian motion stopped at some properly chosen random times:

\begin{theoreme}[Ray-Knight]\label{RK1}
Let $a$ be a positive real number.\\
The process $(L(\tau(a),a-y),\ y\in[0,a])$ is a square of a $2$-dimensional Bessel process started at $0$. And conditionally on $L(\tau(a),0)$, $(L(\tau(a),-y),\ y\geq0)$ is a square of a $0$-dimensional Bessel process started at  $L(\tau(a),0)$ independent of $(L(\tau(a),y),\ y\geq0)$.
\end{theoreme}
\begin{theoreme}[Ray-Knight]\label{RK2}
Let $r$ be a positive real number.\\
The processes $(L(\sigma(r,0),y),\ y\in\R^+)$ and $(L(\sigma(r,0),-y),\ y\in\R^+)$ are two independent squares of $0$-dimensional Bessel processes started at $r$.
\end{theoreme}

In the rest of the article, we denote by $Z$ a square of a $0$-dimensional Bessel process started at $1$ and by $Q$ a square of a $2$-dimensional Bessel process started at $0$. To use the previous theorems, we have to estimate the behavior of these two processes.
\begin{lemme}\label{tal}
For all $v,\delta,M>0$, we have
\begin{enumerate}[(i)]
\item $\displaystyle \Pp\left(\sup_{0\leq t\leq v} |Z(t)-1|\geq\delta\right)\leq 4\frac{\sqrt{(1+\delta)v}}{\delta}\exp\left(-\frac{\delta^2}{8(1+\delta)v}\right)$,\label{tal1}
\item $\displaystyle \Pp\left(\sup_{t\geq0} Z(t)\geq M\right)=\frac{1}{M}$,\label{tal3}
\item $\displaystyle \Pp\left(\sup_{0\leq t\leq v} Q(t)\geq M\right)\leq 4e^{-\frac{M}{2v}}$,\label{tal4}
\item There is a constant $K>0$ such that for any $0<a<b$, 
$$\Pp\left(\sup_{a\leq t\leq b} \frac{1}{t}Q(t)\geq M\right)\leq Ke^{-\frac{M}{2\log(8b/a)}}.
$$\label{tal5}
\end{enumerate}
\end{lemme}
\begin{proof}
$(\ref{tal1})$ and $(\ref{tal3})$ are proved in \cite{Talet} Lemma 3.1 (the results are stated for a Bessel process but they are actually true for a \emph{squared} Bessel process).
\begin{enumerate}
\item[$(\ref{tal4})$] Denote by $B$ and $\tilde{B}$ two independent Brownian motions. The processes $Q$ and $B^2+\tilde{B}^2$ have the same law, so
\begin{align*}
 \Pp\left(\sup_{0\leq t\leq v} Q(t)\geq M\right)&\leq \Pp\left(\left(\sup_{0\leq t\leq v} |B|(t)\right)^2+\left(\sup_{0\leq t\leq v} |\tilde{B}|(t)\right)^2\geq M\right)\\
&\leq 4\Pp\left(\left(\sup_{0\leq t\leq v}B(t)\right)^2+\left(\sup_{0\leq t\leq v}\tilde{B}(t)\right)^2\geq M\right)
\end{align*}
where the second inequality comes from the reflection principle. For a fixed $v$, $\sup_{0\leq s\leq v}B(s)\egloi|B(v)|$, then 
\begin{align*}
\Pp\left(\sup_{0\leq t\leq v} Q(t)\geq M\right)&\leq4\Pp\left(Q(v)\geq M\right)=4e^{-\frac{M}{2v}},
\end{align*}
as $Q(v)$ has exponential distribution of mean $2v$.
\item[$(\ref{tal5})$] Thanks to the scaling property of $Q$,
\begin{align*}
\Pp\left(\sup_{a\leq t\leq b} \frac{1}{t}Q(t)\geq M\right)&=\Pp\left(\sup_{a/b\leq t\leq 1} \frac{1}{t}Q(t)\geq M\right)\\
&\leq \Pp\left(\sup_{0\leq t\leq 1} \sqrt{\frac{Q(t)}{t\log(8/t)}}\geq \sqrt{\frac{M}{\log(8b/a)}}\right).
\end{align*}
We then conclude with Lemma $6.1$ in \cite{HuShi2} which says that there exist a constant $K>0$ such that for all $x>0$,
$$
 \Pp\left(\sup_{0\leq t\leq 1} \sqrt{\frac{Q(t)}{t\log(8/t)}}\geq x\right)\leq Ke^{-\frac{x^2}{2}}.
$$
\end{enumerate}
\end{proof}

We also need to study the behavior of the environment $W$. We start with a notation for the minimum of $W$ on an interval
 $$\Wu(x,y):=\left\{\begin{array}{lc}
              \min_{z\in[x,y]}W(z)&\textrm{if }x\leq y\\
	      +\infty&\textrm{if not}
             \end{array}\right.,
$$
another one for the maximum
$$\Wb(x,y):=\left\{\begin{array}{lc}
              \max_{z\in[x,y]}W(z)&\textrm{if }x\leq y\\
	      -\infty&\textrm{if not}
             \end{array}\right.$$
and a last one for the environment reversed in time 
$$\left(\hW(x),\ x\in\R\right):=\left(W(-x),\ x\in\R\right).$$
Define  now
\begin{align}
H_{v}&:=\inf\{x\geq0\ ;\ W(x)-\Wu(0,x)\geq{v}\},\\
m_{v}&:=\inf\{x\geq0\ ;\ W(x)=\Wu(0,H_{v})\}\label{defmv}
\end{align}
and $\widehat{H}_{v}$ and $\widehat{m}_v$ the corresponding points for $\hW$. Brox showed in \cite{Brox} that at time $e^{v}$, with high probability, the process $\X$ has spent much of its time in the neighborhood of $m_{v}$ or of $\widehat{m}_v$. For our study, we need to know the law of the environment in the neighborhood of these points; it is given by the following theorem due to Tanaka (Lemma 3.1 in \cite{Tanaka}, see also the proposition page 164 in \cite{Tanaka3}).
\begin{theoreme}\label{thtanaka}
Let $R$ be a Bessel process of dimension $3$ started at $0$ and define
\begin{align*}
 \tau_R({v})&:=\inf\{x\geq0/R(x)\geq{v}\},\\
\zeta_R({v})&:=\inf\{x\geq0/R(x)-\inf_{y\geq x}R(y)\geq{v}\}\textrm{ and}\\
\rho_R({v})&:=\sup\{x\leq\zeta_R({v})/R(x)-\inf_{y\geq x}R(y)=0\}.
\end{align*}
Under $\Pmil$, the process $\left(W(-x+m_{v})-W(m_{v}),\ x\in[0,m_{v}]\right)$ and the process $\left(W(x+m_{v})-W(m_{v}),\ x\in[0,H_{v}-m_{v}]\right)$ are independent and the following equalities in law hold:
 $$\big(W(-x+m_{v})-W(m_{v}),\ x\in[0,m_{v}]\big)\egloi\big(R(x),\ x\in[0,\rho_R({v})]\big)$$ 
and
$$\big(W(x+m_{v})-W(m_{v}),\ x\in[0,H_{v}-m_{v}]\big)\egloi\big(R(x),\ x\in[0,\tau_R({v})]\big).$$
\end{theoreme}
Therefore, to use this theorem it is necessary to have informations on the behavior of Bessel processes of dimension $3$.
\begin{lemme}\label{lemboro}
Let $R$ be a $3$-dimensional Bessel process started in $0$. There is a positive real number $K$ such that for every $a,x>0$ 
,\begin{enumerate}[(i)]
  \item \label{lembo1}$\displaystyle
\frac{a}{\sqrt{x}}e^{-a^2/2x}\leq \Pp\left(\sup_{[0,x]}R>a\right) 
\leq K\left(\frac{a}{\sqrt{x}}+\frac{\sqrt{x}}{a}\right)e^{-a^2/2x}.
$
\item \label{lembo2}$\displaystyle \Pp\left(\sup_{[0,x]}R<a\right)\geq\frac{1}{K}e^{-\pi^2x/(2a^2)}$\\
\item \label{lembo3}$\displaystyle \Pp\left(\int_0^\infty e^{-R(x)}\ud x> a\right)\leq Ke^{-j_0^2a/8}$ where $j_{0}$ is the smallest strictly positive root of Bessel function $J_0$.
 \end{enumerate}
\end{lemme}
\begin{proof}
\begin{enumerate}[\it (i)]
 \item According to the reflection principle for Brownian motion, one can find a $K>0$ such that $$\Pp\left(\sup_{[0,x]}R>a\right)\leq K \Pp\left(R(x)>a\right).$$
Moreover, $\Pp\left(\sup_{[0,x]}R>a\right)\geq \Pp\left(R(x)>a\right)$. So Item $(\ref{lembo1})$ of the lemma is a consequence of usual estimates for $3$-dimensional Bessel processes.
\item Recall the Bessel function of the first kind (see \cite{AbrSte} chapter $9$) 
$$J_{1/2}(x)=\sqrt{\frac{2}{\pi x}}\sin x,
$$
its smallest positive root is $\pi$. Then, according to Theorem $2$ of \cite{CieTay}, there is a positive number $K$ such that
$$
 \Pp\left(T_R(1)\geq x\right)\sim \frac{1}{K}e^{-\pi^2x/2}.
$$
Then, (the value of $K$ can change)
$$
 \Pp\left(T_R(1)\geq x\right)\geq\frac{1}{K}e^{-\pi^2x/2}
$$
and
$$
\Pp\left(\sup_{[0,x]}R<a\right)=\Pp\left(T_R(a)\geq x\right)
\geq\frac{1}{K}e^{-\pi^2x/(2a^2)}.
$$
\item Le Gall's Ray-Knight theorem (Proposition 1.1 of \cite{LeGall}) shows that the integral $1/4\int_{0}^{\infty}e^{-R(x)}\ud x$ has the same law as $T_Q(1)$, the hitting time of height $1$ by a squared Bessel process of dimension $2$ started at $0$.
Then according to Theorem $2$ of \cite{CieTay}, as in the proof of the previous item, 
\begin{align*}
\Pp\left(\int_0^\infty e^{-R(x)}\ud x> a\right)=&\Pp\left(T_Q(1)> \frac{a}{4}\right)\leq Ke^{-j^2_{0}a/8}.
\end{align*}
\end{enumerate}
\end{proof}

\section{Estimates on the environment}\label{estenv}
The process $\hW$ has the same law as $W$, this allows to restrict the study to $W$ on $\R^+$ and to get similar results on $\R^-$ by symmetry.

We study the environment on $[0,H_{v}]$, the \emph{valley} of height $v$ and particularly in the neighborhood of $m_{v}$ as it is the place where the diffusion spends most of its time. Unfortunately, the probability that at time $e^{v}$, the process has reached the bottom $m_{v}$ and has not left the valley is not growing fast enough to derive almost sure results. So we rather study the valley of height ${v}-c_1\log{v}$, where $c_1$ is a positive real number whose value will be determined later, so that, with high probability, at time $e^{v}$, the process has reached the bottom of this valley and the valley of height ${v}+c_3\log{v}$ so that the process is still inside at time $e^{v}$. We therefore fix three constants $ c_1, c_2, c_3> 0 $ with $c_1\geq c_2$ and define recursively for any ${v}>1$, $\bpd[0]:=0$ and for $i\geq0$,
\begin{align*}
\bpd[i+1]&:=\inf\{x\geq \bmd[i]\ ;\ W(x)-\Wu(\bmd[i],x)\geq {v}-c_1\log{v}\},\\
\mpd[i+1]&:=\inf\{x\geq \bmd[i]\ ;\ W(x)=\Wu(\bmd[i],\bpd[i+1])\}.
\end{align*}

Denote also for any $i\in\N^*$,
\begin{align*}
	\amd[i]&:=\sup\{x\leq \mmd[i]\ ; \ W(x)-W(\mmd[i])\geq {v}-c_2\log{v}\}\vee\bmd[i-1],\\
	\Bpd&:=\inf\{x\geq 0\ ; \ W(x)-\Wu(0,x)\geq {v}+c_3\log{v}\},\\
	\ma&:=\inf\{x\geq0\ ; \ W(x)=\Wu(0,\Bpd)\},\\
	\Cma&:=\inf\{x\geq \ma\ ; \ W(x)-W(\ma)\geq {v}-c_1\log{v}\}\textrm{ and}\\
	\Dma&:=\sup\{x\leq \ma\ ; \ W(x)-W(\ma)\geq {v}-c_2\log{v}\}\vee0.
\end{align*}
\begin{figure}[ht]%
\begin{picture}(0,0)%
\includegraphics{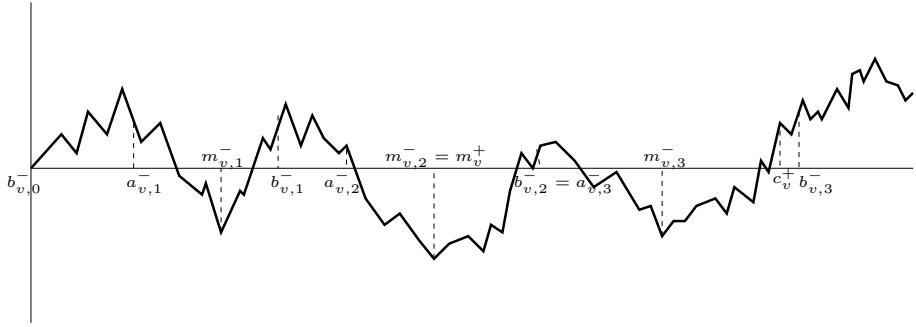}%
\end{picture}%
\setlength{\unitlength}{2072sp}%
\begingroup\makeatletter\ifx\SetFigFontNFSS\undefined%
\gdef\SetFigFontNFSS#1#2#3#4#5{%
  \reset@font\fontsize{#1}{#2pt}%
  \fontfamily{#3}\fontseries{#4}\fontshape{#5}%
  \selectfont}%
\fi\endgroup%
\begin{picture}(10713,3849)(1696,-5338)
\put(1650,-3706){\makebox(0,0)[lb]{\smash{{\SetFigFontNFSS{6}{7.2}{\rmdefault}{\mddefault}{\updefault}{\color[rgb]{0,0,0}$b_{{v},0}^-$}%
}}}}
\put(9181,-3391){\makebox(0,0)[lb]{\smash{{\SetFigFontNFSS{6}{7.2}{\rmdefault}{\mddefault}{\updefault}{\color[rgb]{0,0,0}$m_{{v},3}^-$}%
}}}}
\put(3061,-3706){\makebox(0,0)[lb]{\smash{{\SetFigFontNFSS{6}{7.2}{\rmdefault}{\mddefault}{\updefault}{\color[rgb]{0,0,0}$a_{{v},1}^-$}%
}}}}
\put(3950,-3391){\makebox(0,0)[lb]{\smash{{\SetFigFontNFSS{6}{7.2}{\rmdefault}{\mddefault}{\updefault}{\color[rgb]{0,0,0}$m_{{v},1}^-$}%
}}}}
\put(4771,-3706){\makebox(0,0)[lb]{\smash{{\SetFigFontNFSS{6}{7.2}{\rmdefault}{\mddefault}{\updefault}{\color[rgb]{0,0,0}$b_{{v},1}^-$}%
}}}}
\put(5400,-3706){\makebox(0,0)[lb]{\smash{{\SetFigFontNFSS{6}{7.2}{\rmdefault}{\mddefault}{\updefault}{\color[rgb]{0,0,0}$a_{{v},2}^-$}%
}}}}
\put(6121,-3391){\makebox(0,0)[lb]{\smash{{\SetFigFontNFSS{6}{7.2}{\rmdefault}{\mddefault}{\updefault}{\color[rgb]{0,0,0}$m_{{v},2}^-=m_{v}^+$}%
}}}}
\put(7651,-3706){\makebox(0,0)[lb]{\smash{{\SetFigFontNFSS{6}{7.2}{\rmdefault}{\mddefault}{\updefault}{\color[rgb]{0,0,0}$b_{{v},2}^-=a_{{v},3}^-$}%
}}}}
\put(10711,-3661){\makebox(0,0)[lb]{\smash{{\SetFigFontNFSS{6}{7.2}{\rmdefault}{\mddefault}{\updefault}{\color[rgb]{0,0,0}$c_{v}^+$}%
}}}}
\put(11026,-3706){\makebox(0,0)[lb]{\smash{{\SetFigFontNFSS{6}{7.2}{\rmdefault}{\mddefault}{\updefault}{\color[rgb]{0,0,0}$b_{{v},3}^-$}%
}}}}
\end{picture}%
\caption{A sample path of $W$}
\end{figure}

Obviously, there is a $i\in\N^*$ such that $\ma=\mpd[i]$. We want to prove that, with a probability large enough, $\ma\in\{\mpd[1],\ \mpd[2]\}$ and moreover that, in the valley $[0,\Bpd]$, the points after $\Cma$ are higher than $W(\ma)+(c_1+c_3)\log{v}$. This can be expressed formally as follows:
\begin{align*}
	\Gamma_{v}^{1}&:=\left\{\Bpd\leq\bpd[3]\ ;\ \Wu(\Cma,\Bpd)-W(\ma)\geq(c_1+c_3)\log{v}\right\}.
\end{align*}
We also need many more technical conditions to ensure that the environment does not stray too far from its average behavior:
\begin{align*}
	\Gamma_{v}^{2}&:=\left\{\bpd[3]\leq{v}^6\ ;\ W(\mmd)\geq-{v}^2\ ;\ W(\mpd)-W(\bmd)\geq-{v}^2\right\}\\
&\bigcap\left\{\mmd-\amd\geq\frac{1}{{v}^2}\ ;\ \mpd-\apd\geq\frac{1}{{v}^2}\right\}\\
&\bigcap\left\{\Bpd-\ma\geq{v}\ ;\ W(\Bpd)-\Wu((\Bpd-\log{v})\vee\ma,\Bpd)\leq2\log{v}\right\}\\
&\bigcap\left\{\Wb((\mmd-\log{v})\vee\amd,\mmd)-W(\mmd)\leq2\log{v}\right\}\\
&\bigcap\left\{\Wb((\mpd-\log{v})\vee\apd,\mpd)-W(\mpd)\leq2\log{v}\right\}.
\end{align*}

We would also wish sometimes that $\ma=\mpd[1]$ (which of course is obtained with a probability lower than the previous one). For that we use the event:
\begin{align*}
 	\Gamma_{v}^{3}&:=\left\{\Bpd\leq\bpd[2]\ ;\ \Wu(\Cma,\Bpd)-W(\ma)\geq(c_1+c_3)\log{v}\right\}.
\end{align*}
Define similarly $\widehat{\Gamma}_{v}^{1}$, $\widehat{\Gamma}_{v}^{2}$ and $\widehat{\Gamma}_{v}^{3}$ from $\hW$.

We will therefore work on 
\begin{equation}
  \Gamma_{v}=\genv  \label{defGv}
\end{equation}
 or on
\begin{equation}
\Gamma_{v}'=\penv. \label{defGvp}
\end{equation}
 The first step, as stated before, is to show that these events occur with a high enough probability.
 We denote for every event $A$, $\overline{A}:=\Omega\setminus A$.

\begin{proposition}\label{probGamma}  There exists a constant $K>0$ such that for ${v}$ large enough,

$$\Pmil(\overline{\Gamma}_{v})\leq K\left(\frac{\log{v}}{{v}-c_1\log{v}}\right)^2\textrm{ and }\ \Pmil(\overline{\Gamma'}_{v})\leq \frac{K\log{v}}{{v}-c_1\log{v}}.$$
\end{proposition}
\begin{proof}
Start with the upper bound for $\Pmil(\overline{\Gamma}^{1}_{v})$. Define
\begin{align*}
 W_1&:=(W(\bmd+x)-W(\bmd))_{x\in[0,\bpd-\bmd]}\textrm{ and }\\
W_2&:=(W(\bpd+x)-W(\bpd))_{x\in[0,\bpd[3]-\bpd]}.
\end{align*}
The event $\left\{\Bpd>\bd\right\}$ is included in
\begin{displaymath}
\ \left\{\sup_{[0,\bpd-\bmd]}W_1\leq (c_1+c_3)\log{v}\ ;\ \sup_{[0,\bpd[3]-\bpd]}W_2\leq (c_1+c_3)\log{v}\right\}.
\end{displaymath}

The processes $W_1$ et $W_2$ are independent and are distributed like $$\left(W(x)\ ,\ x\in[0,\bmd]\right).$$
Therefore,
\begin{align*}
 \Pmil\left(\Bpd>\bd\right)&\leq\left(\Pmil\left(\sup_{[0,\bmd]}W\leq (c_1+c_3)\log{v}\right)\right)^2\\
&=\left(\Pmil\left(W\textrm{ hits }-{v}+(2c_1+c_3)\log{v}\textrm{ before }(c_1+c_3)\log{v}\right)\right)^2\\
&=\left(\frac{(c_1+c_3)\log{v}}{{v}-c_1\log{v}}\right)^2.
\end{align*}
 	
For the second part of $\overline{\Gamma}^{1}_{v}$, according to Theorem \ref{thtanaka}, if we denote by $R$ a Bessel process of dimension $3$ started at ${v}-c_1\log{v}$, then
\begin{align*}
\Pmil\Big(\Wu(\Cma,\Bpd)-W(\ma)<(c_1+&c_3)\log{v}\Big)\\
=&\Pmil\left(\underline{R}(0,\tau_R({v}+c_3\log{v}))<(c_1+c_3)\log{v}\right)\\
=&\left(\frac{(c_1+c_3)\log{v}}{{v}-c_1\log{v}}\right)^2.
\end{align*}
For the last equality, see for example Property $2.2.2$ of part II, chap $5$ in \cite{Borodin}.
We obtain in the same way
$$\Pmil(\overline{\Gamma}_{v}^3)\leq\frac{2(c_1+c_3)\log{v}}{{v}-c_1\log{v}}.$$

Continue with an upper bound for $\Pmil(\overline{\Gamma}_{v}^2)$. As $W-\Wu$ has the same law as $|W|$, 
\begin{align*}
 \Pmil(\bpd[3]>{v}^6)&\leq\Pmil(\bpd[3]-\bpd>\frac{{v}^6}{3})+\Pmil(\bpd-\bmd>\frac{{v}^6}{3})+\Pmil(\bmd>\frac{{v}^6}{3})\\
&=3\Pmil(\tau_{|W|}({v}-c_1\log{v})>\frac{{v}^6}{3})\\
&\leq 3\Pmil(\tau_{W}({v})>\frac{{v}^6}{3})
\leq\frac{K}{{v}^2}.
\end{align*}
Moreover, $-W(\mmd)$ and $W(\bmd)-W(\mpd)$ are exponentially distributed with mean ${v}-c_1\log{v}$ (see for example the first lemma of \cite{NevPit}) and are independent. Thus 
\begin{align*}
\Pmil\left(W(\mpd)-W(\bmd)<-{v}^2\right)&=\Pmil\left(W(\mmd)<-{v}^2\right)\\
&\leq \Pmil\left(W(\mmd)<-({v}-c_1\log{v})^2\right)\\
&={v}^{c_1}e^{-{v}}.
\end{align*}

Thanks to Theorem \ref{thtanaka}, still denoting by $R$ a Bessel process of dimension $3$ but now started at $0$, as $\amd[i]\geq\bmd[i-1]$,
\begin{align*}
\Pmil\left(\mmd[i]-\amd[i]<\frac{1}{{v}^2}\right)\leq&\Pmil\left(\mmd[i]-\amd[i]<\frac{1}{({v}-c_1\log{v})^2}\right)\\
\leq&\Pmil\left(\tau_{R}({v}-c_1\log{v})<\frac{1}{({v}-c_1\log{v})^2}\right)\\
&+\Pmil\left(\mmd[i]-\bmd[i-1]<\frac{1}{({v}-c_1\log{v})^2}\right).
\end{align*}
Yet, according to the scaling property of Brownian motion and a lemma proved by Cheliotis in \cite{Cheliotis} (claim at the end of the proof of Lemma $13$), there is a constant $K>0$, such that
\begin{align*}
\Pmil\left(\mmd[i]-\bmd[i-1]<\frac{1}{({v}-c_1\log{v})^2}\right)&=\Pmil\left(m_{1}<\frac{1}{({v}-c_1\log{v})^4}\right)\\
&\leq \frac{K}{({v}-c_1\log{v})^2}.
\end{align*}
Moreover, Item $(\ref{lembo1})$ of Lemma \ref{lemboro} gives
\begin{align*}	\Pmil\left(\tau_{R}({v}-c_1\log{v})<\frac{1}{({v}-c_1\log{v})^2}\right)
	&\leq K({v}-c_1\log{v})^2e^{-({v}-c_1\log{v})^4/2}.\label{eq3}
\end{align*}
We also obtain the following upper bound:
\begin{align*}
\Pmil\left(\Bpd-\ma<{v}\right)&=\Pmil\left(\tau_R({v}+c_3\log{v})<{v}\right)\\
&\leq\Pmil\left(\tau_R({v})<{v}\right)\\
&\leq K\sqrt{{v}}e^{-{v}/2}.
\end{align*}
It remains to control $\Pmil(W(\Bpd)-\Wu((\Bpd-\log{v})\vee\ma,\Bpd)>2\log{v})$. Let $\beta={v}+c_3\log{v}$. Using one more time Theorem \ref{thtanaka}, we see that
\begin{align*}
W(\Bpd)&-\Wu((\Bpd-\log{v})\vee\ma,\Bpd)\\
&=W(\Bpd)-W(\ma)-\left(\min_{t\in[0,\log{v}\wedge(\Bpd-\ma)]}W((\Bpd-t)-W(\ma)\right)
\end{align*}
has the same law as
$$\beta-\min_{t\in[0,\log{v}\wedge\tau_R(\beta)]}R(\tau_R(\beta)-t)=\max_{t\in[0,\log{v}\wedge\tau_R(\beta)]}(\beta-R(\tau_R(\beta)-t)).$$
And according to Proposition 4.8, Chapter VII of \cite{RevYor}, the processes $$\left(\beta-R(\tau_R(\beta)-t),\ t\in[0,\tau_R(\beta)]\right)\textrm{ and }\left(R(t),\ t\in[0,\tau_R(\beta)]\right)$$ have the same law. Therefore,
\begin{align*}
\Pmil\Big(W(\Bpd)-\Wu((\Bpd-\log{v})\vee\ma,&\Bpd)>2\log{v}\big)\\
&=\Pmil\left(\max_{t\in[0,\log{v}\wedge\tau_R(\beta)]}R(t)>2\log{v}\right)
\end{align*}
and Item $(\ref{lembo1})$ of Lemma \ref{lemboro} implies
\begin{align*}
\Pmil\left(\max_{t\in[0,\log{v}]}R(t)>2\log{v}\right)\leq K\frac{\sqrt{\log{v}}}{{v}^2}.
\end{align*}
 Finally, we just have to obtain an upper bound for $$\Pmil(\Wb((\mmd-\log{v})\vee\amd,\mmd)-W(\mmd)>2\log{v})$$ to prove the proposition. It can be obtained as the previous one.
\end{proof}

We now come back to the local time of the diffusion $\X$.

\section{Asymptotic behavior of $\LX$ at particular times}\label{briques}

Let $r$ be a positive real number.  As for the numbers $c_i$, its value will be fixed later. Define $\smd:=\tla(re^{{v}},\mmd)$ the inverse of local time in $\mmd$ and in the same way $\spd$, $\smg$, $\spg$, $\tla^+_{v}$ and $\widehat{\tla}^+_{v}$.
At these times, it is possible to estimate the local time of $\X$ in the neighborhood of the corresponding point. We first give an estimate at a fixed environment in Proposition \ref{LT}, then in Proposition \ref{majind}, the estimate is independent of the environment provided that this one belongs to $\Gamma_{v}$.
\begin{proposition}
\label{LT}
Define for $i\in\left\{1,2\right\}$ and $0<\delta<1$,
\begin{align*}
	\Appm&:=\left\{\forall x\in[\apmd,\bpmd], \left|\frac{\LXpmd}{re^{{v}-W(x)+W(\mpmd)}}-1\right|\leq \delta
\right\},\\
	\Bi&:=\left\{\forall x\in[\bim,\apmd),\LXpmd\leq\delta re^{{v}}\right\},\\
	\C_{v}&:=\left\{\forall x\in[\Cma,\Bpd],\LX(\tla^+_{v},x)\leq\delta re^{{v}}\right\}\textrm{ and}\\
	\D_{v}&:=\left\{\forall x>\Bpd, \LX(\tla^+_{v},x)=0\right\}
\end{align*}
and in the same way $\widehat{\Appm}$, $\widehat{\Bi}$, $\widehat{\C_{v}}$ and $\widehat{\D_{v}}$ from $\widehat{W}$.
There is a constant $K>0$ such that for ${v}$ large enough, 
for any $0\leq\delta\leq1$ and any $r>0$,
\begin{align*}
	\PW\left(\overline{\Appm}\right)&\leq\frac{K}{\delta}\sqrt{\frac{\bpmd}{r{v}^{c_2}}}\exp{\left(-\frac{\delta^2  r{v}^{c_2}}{K\bpmd}\right)},\\
	\PW\left(\overline{\Bi}\right)&\leq K\exp{\left(-\frac{\delta r{v}^{c_1}}{4(\mpmd-\bim)\log\left(8\frac{S(\mpmd)-S(\bim)}{(S(\mpmd)-S(\apmd))}\right)}\right)}+\frac{2}{\delta {v}^{c_1-c_2}},\\
	\PW\left(\overline{\C_{v}}\right)&\leq\frac{1}{\delta}e^{-\Wu(\Cma,\Bpd)+W(\ma)}\textrm{ and}\\
	\PW\left(\overline{\D_{v}}\right)&\leq\frac{re^{{v}+W(\ma)}}{2(S(\Bpd)-S(\ma))}.
\end{align*}	
Similar estimates hold for $\PW(\widehat{\Appm})$, $\PW(\widehat{\Bi})$, $\PW(\widehat{\C_{v}})$ and $\PW(\widehat{\D_{v}})$.
\end{proposition}
\begin{proof}
We estimate the probabilities of the events relative to $ W $, the ones relative to $\hW$ follow by symmetry. To simplify notations, all along the proof, we shall not mark the index ${v}$ for variables and events. Begin with the events $\A$ and $\B$ ($\mathcal{A}^2$ and $\mathcal{B}^2$ can be studied in the same way).

The local time can be decomposed in two terms. The first one represents the contribution of the local time before $\tm$ (the first time where $\X$ reaches $\mm$) and is negligible compared to the second one which represents the contribution of the local time between $\tm$ and $\sam=\tlam$:
\begin{equation}\label{coupe}
\LXm=\LX(\tm,x)+\left(\LXm-\LX(\tm,x)\right).
\end{equation}
The following lemma describes the behavior of the first term.
\begin{lemme}\label{lemtau}
For any $r,v>0$ and $0<\delta<1$,
\begin{align*}
\PW_{\A,1}&:=\PW\left(\sup_{x\in[\am,\mm]}\frac{\LX(\tm,x)}{re^{{v}-(W(x)-W(\mm))}}>\delta\right)\\
&\leq K
\exp{\left(-\frac{\delta r{v}^{c_2}}{2\mm}\right)},\\
\PW_{\B,1}&:= \PW\left(\sup_{x\in[\b,\am)}\LX(\tm,x)>\delta re^{{v}}\right)\\
&\leq K\exp{\left(-\frac{\delta r{v}^{c_1}}{2(m^-_1-b^-_0)\log\left(8\frac{S(m^-_1)-S(b^-_0)}{S(m^-_1)-S(a^-_1)}\right)}\right)}.
\end{align*}
\end{lemme}
To highlight the fact that the computations are identical for $\B$ and for $\mathcal{B}^2$ we make the quantity $\b$ appears, although it is zero.
\begin{proof}
Thanks to (\ref{eqx}), it is easy to verify that $\tm=T(\tau(S(\mm))$. So using $(\ref{eqLx})$, for any $x\geq0$,
\begin{align*}
\LX(\tm,x)&=e^{-W(x)}L(\tau(S(\mm)),S(x)).
\end{align*}
Then,
$$
\PW_{\A,1}=\PW\left(\sup_{x\in[\am,\mm]}L(\tau(S(\mm)),S(x))>\delta re^{{v}+W(\mm)}\right).
$$
According to the first Ray-Knight theorem, Theorem \ref{RK1}, 
$$\left(L(\tau(S(\mm)),S(\mm)-y),\ y\in[0,S(\mm)]\right)$$
 is distributed as a squared Bessel process of dimension $2$ started at $0$. Therefore, with Item $(\ref{tal4})$ of Lemma \ref{tal},
$$
\PW_{\A,1}\leq K\exp{\left(-\frac{\delta re^{{v}+W(\mm)}}{2(S(\mm)-S(\am))}\right)}
$$
and by definition of $\am$,
\begin{align*}
	S(\mm)-S(\am)&=\int_{\am}^{\mm} e^{W(x)}\ud x\leq \mm e^{\Wb(\am,\mm)}\\
	&\leq\mm e^{{v}-c_2\log{v}+W(\mm)}.
\end{align*}
Hence the first upper bound of the lemma is obtained. 

Continue with the second one: using a similar argument and denoting by $Q$ a squared Bessel process of dimension $2$ started at $0$, we get
\begin{align*}
\PW_{\B,1}&=\PW\left(\sup_{x\in[\b,\am)}e^{-W(x)}Q(S(\mm)-S(x))>\delta re^{{v}}\right)\\	
	&= \PW\left(\sup_{x\in[\b,\am)}\frac{e^{-W(x)}(S(\mm)-S(x))}{S(\mm)-S(x)}Q(S(\mm)-S(x))>\delta re^{{v}}\right).
\end{align*}
By definition of $\mm$, for any $x\in[\b,\am)$,
\begin{align*}
	e^{-W(x)}(S(\mm)-S(x))
	\leq (\mm-\b) e^{-W(x)+\Wb(x,\mm)}
\end{align*}
As $\bm$ is the first positive number $x$ such that $W(x)-\Wu(\b,x)\geq {v}-c_1\log{v}$,
\begin{align*}
(\mm-\b) e^{-W(x)+\Wb(x,\mm)}\leq(\mm-\b) e^{{v}-c_1\log{v}}.
\end{align*}
Thus, coming back to the probability $\PW_{\B,1}$, we obtain 
\begin{align*}
\PW_{\B,1}\leq \PW\left(\sup_{u\in[S(\mm)-S(\am),S(\mm)-S(\b))}\frac{1}{u}Q(u)>\frac{\delta r{v}^{c_1}}{\mm-\b}\right).
\end{align*}
According to Item $(\ref{tal5})$ of Lemma $\ref{tal}$, we finally have
\begin{align*}
\PW_{\B,1}&\leq K\exp{\left(-\frac{\delta r{v}^{c_1}}{2(\mm-\b)\log\left(8\frac{S(\mm)-S(\b)}{S(\mm)-S(\am)}\right)}\right)}.
\end{align*}
This concludes the proof of the lemma.
\end{proof}
Now, we study the second term of \eqref{coupe}.
\begin{lemme}\label{lemsig}
For any $r>0$, ${v}\geq1$ and $0<\delta<1$,
\begin{align*}
\PW_{\A,2}&:=\PW\left(\sup_{x\in[\am,\bm]} \left|\frac{\LXm-\LX(\tm,x)}{re^{{v}-(W(x)-W(\mm))}}-1\right|>\delta \right)\\
&\leq \frac{8}{\delta}\sqrt{\frac{(1+\delta)\bm}{r{v}^{c_2}}}\exp\left(-\frac{\delta^2r{v}^{c_2}}{8(1+\delta )\bm}\right),\\
\PW_{\B,2}&:=\PW\left(\sup_{x\in[\b,\am)}\LXm-\LX(\tm,x)>\delta re^{{v}}\right)\\
&\leq\frac{1}{\delta{v}^{c_1-c_2}}.
\end{align*}
\end{lemme}
\begin{proof}
It is easy to verify that the inverse of local time $\tla$ satisfies the following equality for every $ r> 0 $ and $y\in\R$,
\begin{equation}
	\tla(r,y)=T(\sigma(re^{W(y)},S(y))).\label{eqformsa}
\end{equation}
Thus, thanks to $(\ref{eqLx})$, for any $r>0$ and $y\in\R$,
\begin{align*}
\LX(\sam,x)&=e^{-W(x)}L(\sigma(re^{W(\mm)+{v}},S(\mm)),S(x)).
\end{align*}
And so the following expression for the local time holds,
\begin{align*}
\LXm&-\LX(\tm,x)\\
&=e^{-W(x)}\left(L(\sigma(re^{W(\mm)+{v}},S(\mm)),S(x))-L(\tau(S(\mm)),S(x))\right)\\
&\egloi e^{-W(x)}re^{W(\mm)+{v}}L(\sigma(1,0),s(x))
\end{align*}
where
\begin{displaymath}
s(x):=(S(x)-S(\mm))\frac{e^{-W(\mm)-{v}}}{r}.
\end{displaymath}
Denote by $Z$ the square of a Bessel process of dimension $0$ started at $1$. According to the second Ray-Knight theorem (Theorem \ref{RK2}), we have
\begin{align*}
\PW_{\A,2}&\leq \PW\left(\sup_{0\leq y\leq|s(\am)|} \left|Z(y)-1\right|>\delta \right)\\
&+\PW\left(\sup_{0\leq y\leq s(\bm)} \left|Z(y)-1\right|>\delta \right).
\end{align*}
Therefore, using Item $(\ref{tal1})$ of Lemma \ref{tal},
\begin{align*}
 \PW_{\A,2}&\leq \frac{4}{\delta }\sqrt{(1+\delta)|s(\am)|}\exp\left(-\frac{\delta^2}{8(1+\delta)|s(\am)|}\right)\\
&+ \frac{4}{\delta }\sqrt{(1+\delta)s(\bm)}\exp\left(-\frac{\delta^2}{8(1+\delta)s(\bm)}\right).
\end{align*}
Moreover, the definition of $\am$ implies that
\begin{align*}
|s(\am)|&=\frac{e^{-W(\mm)-{v}}}{r}\int_{\am}^{\mm}e^{W(x)}\ud x\leq\frac{\mm}{r}e^{W(\am)-W(\mm)-{v}}\leq\frac{\bm}{r{v}^{c_2}}.
\end{align*}
We get likewise $|s(\bm)|\leq\bm/(r{v}^{c_1})$.
As $c_1\geq c_2$, these last two inequalities lead to the first point of the lemma. We now prove the second inequality of the lemma. If $\b=\am$, we have obviously $\PW_{\B,2}=0$ else, reasoning in the same way as before, we obtain
\begin{align*}
\PW_{\B,2}&=\PW\left(\sup_{x\in[\b,\am)}e^{-W(x)+W(\mm)}L(\sigma(1,0),s(x))>\delta \right).
\end{align*}
One more time, thanks to the second Ray-Knight theorem and denoting by $Z$ a squared Bessel process of dimension $0$ started at $1$, we obtain
\begin{align*}
\PW_{\B,2}&\leq \PW\left(e^{-\Wu(\b,\am)+W(\mm)}\sup_{u\geq0} Z(u)>\delta \right)\\
&=\frac{1}{\delta}e^{-\Wu(\b,\am)+W(\mm)}.
\end{align*}
The second line is a consequence of Item $(\ref{tal3})$ of Lemma \ref{tal}.
Denote by $\nm$ the unique real number in $[\b,\am]$ such that $W(\nm)=\Wu(\b,\am)$ then
\begin{align*}
W(\mm)-W(\nm)&=\Wb(\nm,\am)-W(\nm)-(\Wb(\nm,\am)-W(\mm))\\
&\leq {v}-c_1\log{v}-({v}-c_2\log{v})=(c_2-c_1)\log{v}.
\end{align*}
Therefore,
\begin{displaymath}
 \PW_{\B,2}\leq\frac{1}{\delta {v}^{c_1-c_2}}
\end{displaymath}
and the lemma is proved.
\end{proof}
Combining the results of Lemmas \ref{lemtau} and \ref{lemsig} yields to the upper bounds of $\PW(\overline{\A})$ and $\PW(\overline{\B})$
of Proposition \ref{LT}.

We continue with the estimate of $\PW(\C
)$. Reducing once again the local time of $\X$ to the local time of a Brownian motion by a time and space change, we obtain
$$\PW(\overline{\C})=\PW\left(\sup_{x\in[\Cm,\Bp]}e^{-W(x)+W(\m)}L(\sigma(1,0),s(x))>\delta\right).$$
where $s(x)$ is the same as before but $\mm$ is replaced by $\m$.
One more time, $ Z $ denotes a squared $0$-dimensional Bessel process started at $1$ and the second Ray-Knight theorem gives:
\begin{align*}
\PW(\overline{\C})&\leq \PW\left(e^{-\Wu(\Cm,\Bp)+W(\m)}\sup_{u\geq0} Z(u)>\delta\right).
\end{align*}
So Item $(\ref{tal3})$ of Lemma \ref{tal} yields to the upper bound of Proposition \ref{LT}.
 
Finally, we show that with a high probability diffusion $\X$ does not hit $\Bp$ before time $\tla^+$. I.e. we find an upper bound for $\PW(\overline{\D})$. The scale change in time and space of $\X$ and the usual properties of Brownian motion give (see e.g. \cite{Borodin} Formula 4.1.2 page 185)
\begin{align*}
\PW\left(\overline{\D}\right)&=\PW\Big(\tX(\Bp)<\tla\left(re^{v},\m\right)\Big)\\
&=\PW\left(\tB(S(\Bp))<\sB\left(re^{{v}+W(\m)},S(\m)\right)\right)\\
&=1-\exp{\left(-\frac{re^{{v}+W(\m)}}{2(S(\Bp)-S(\m))}\right)}\leq\frac{re^{{v}+W(\m)}}{2(S(\Bp)-S(\m))}.
\end{align*}
This completes the proof of the proposition.
\end{proof}
We now give upper bounds independent of the environment provided that it is in the set $\Gamma_{v}$ defined in \eqref{defGv}.
\begin{proposition}\label{majind}
We use the same notations as in the previous proposition.
There is a constant $K>0$ such that for ${v}$ large enough, for any $0\leq\delta\leq1$, and any $r>0$, if $W\in\Gamma_{v}$, for $i\in\{1,2\}$,
\begin{align*}
	\PW\left(\overline{\Appm}\right)\leq&\frac{K}{\delta\sqrt{r{v}^{c_2-6}}}\exp{\left(-\frac{\delta^2  r{v}^{c_2-6}}{K}\right)},\\
	\PW\left(\overline{\Bi}\right)\leq&K\exp{\left(-\frac{\delta r{v}^{c_1-8}}{K}\right)}+\frac{2}{\delta {v}^{c_1-c_2}},\\
	\PW\left(\overline{\C_{v}}\right)\leq&\frac{1}{\delta{v}^{c_1+c_3}}\textrm{ and}\\
	\PW\left(\overline{\D_{v}}\right)\leq&\frac{r}{2{v}^{c_3-2}\log{v}}.
\end{align*}	
Once again similar estimates hold for $\PW(\widehat{\Appm})$, $\PW(\widehat{\Bi})$, $\PW(\widehat{\C_{v}})$, $\PW(\widehat{\D_{v}})$.
\end{proposition}
\begin{proof}
We only have to control the values of the upper bounds of Proposition $\ref{LT}$ when $W\in\Gamma_{v}$. For $\PW\left(\overline{\Appm}\right)$, it is enough to notice that, on $\Gamma_{v}$, the variables $\bpmd$ are smaller than ${v}^6$.
Then, we also obtain the upper bound $\mpmd-\bim\leq\bpd\leq{v}^6$. To estimate $\PW\left(\overline{\Bi}\right)$, it remains to study
\begin{align*}
 \frac{S(\mpmd)-S(\bim)}{S(\mpmd)-S(\apmd)}&\leq\frac{(\mpmd-\bim) e^{\Wb(\bim,\mpmd)}}{(\mpmd-\apmd)e^{W(\mpmd)}}.
\end{align*}
First remark that $\Wb(\bim,\mpmd)-W(\bim)\leq{v}-c_1\log{v}\leq{v}$. Then it is easy to see that, on $\Gamma_{v}$, the following inequality holds:
$$
\frac{S(\mpmd)-S(\bim)}{S(\mpmd)-S(\apmd)}\leq{v}^8e^{{v}+{v}^2}.
$$
This implies the second upper bound of the proposition. As on $\Gamma_{v}$, we have $\Wu(\Cma,\Bpd)-W(\ma)\geq(c_1+c_3)\log{v}$, the third estimate is obtained immediately. It remains the upper bound of $\PW\left(\overline{\D_{v}}\right)$. Remark that
\begin{align*}
 S(\Bpd)-S(\ma)
 \geq&\int_{(\Bpd-\log{v})\vee\ma}^{\Bpd}e^{W(x)}\ud x\\
\geq&\left(\log{v}\wedge(\Bpd-\ma)\right)e^{\Wu((\Bpd-\log{v})\vee\ma,\Bpd)}.
\end{align*}
And on $\Gamma_{v}$ we have $\Bpd-\ma\geq{v}$ and $\Wu((\Bpd-\log{v})\vee\m,\Bpd)\geq W(\Bpd)-2\log{v}$, then
\begin{align*}
 \PW\left(\overline{\D_{v}}\right)\leq&\frac{r{v}^2}{2\log{v}}e^{{v}+W(\ma)-W(\Bpd)}.
\end{align*}
As $W(\Bpd)-W(\ma)={v}+c_3\log{v}$, this concludes the proof.
\end{proof}

\section{Asymptotics of local time in deterministic time}\label{supLX}
We fix now the constants $c_i$: take a real number $c> 0$, then $c_1, c_2$ and $c_3$ are chosen as follows $c_1:=2c+8,\ c_2:=c+6\ \textrm{and}\ c_3:=c+2$.
Thanks to Proposition $\ref{majind}$, we can now study the process $\LX$ at the time $$\sigp:=\widehat{\tla}^+_{v}\wedge\smg\wedge\smd\wedge\tla^+_{v}.$$
Define for ${v}$ large enough,  
\begin{align*}
	\Imd:=\int_{\amd}^{\bmd}e^{-W(x)+W(\mmd)}\ud x\ &,\ \Ipd:=\int_{\Dma}^{\Cma}e^{-W(x)+W(\ma)}\ud x
\end{align*}
and define similarly $\Img$ and $\Ipg$ from $\widehat{W}$. Consider finally
\begin{align*}
  \isp:=\Img\wedge\Ipg\wedge\Imd\wedge\Ipd\ &\textrm{et}\ \Isp:=\Img+\Ipg+\Imd+\Ipd.
\end{align*}
Roughly speaking, Proposition \ref{occsig} shows that the process $\sigp/re^{v}$ stays between $\isp$ and $\Isp$.
Moreover, the occupation measure is concentrated in the neighborhood of $\mmd,\ \ma,\ \mmg$ and $\mag$. Precisely, for ${v}$ large enough, define the last time the environment is less than $W(\mmd)+\log 1/\delta$ between $\mmd$ and $\bmd$, 
 $$
 \dmd:=\sup\{\mmd\leq x\leq\bmd\ ,\ W(x)-W(\mmd)\leq \log 1/\delta\}
 $$
 and the first time the environment is less than $W(\mmd)+\log 1/\delta$ between $\amd$ and $\mmd$, 
 $$
 \emd:=\inf\{\amd\leq x\leq\mmd\ ,\ W(x)-W(\mmd)\leq \log 1/\delta\}.
 $$
 Consider then the interval $\Umd:=[\emd,\dmd].$
Define similarly $\dpd$, $\epd$ and $\Upd$ from $\ma$ and the analogous variables for $\widehat{W}$.
At time $\sigp$, the diffusion has spent much of its time in the set 
 $$
 A_{v}:=\Umd\cup\Upd\cup\Umg\cup\Upg.
 $$
 and $\LX^*$ is approximately $re^{v}$.
 \begin{proposition}\label{occsig}
 Define the event
\begin{align*}
 \mathcal{E}_v:=\bigg\{&\nu_{\sigp}(\overline{A}_{v})\leq 4r{v}^{6}e^{v}\delta\ ;\ re^{v}\leq \LX^*(\sigp)\leq re^{v}(1+\delta)\ ;\\ &\isp(1-\delta)\leq\frac{\sigp}{re^{v}}\leq \Isp+2{v}^{6}\delta\bigg\}.
\end{align*}
 There is a constant $K>0$ such that for any $0<\delta<1$ and any $r>0$, if $W\in\Gamma_{v}$,
\begin{align*}
\PW\left(\overline{\mathcal{E}}_v\right)&\leq K\left(\frac{1}{\delta\sqrt{r{v}^{c}}}\exp\left(-\frac{\delta^2r{v}^{c}}{K}\right)+\exp\left(-\frac{\delta r{v}^{c}}{K}\right)+\frac{1}{\delta{v}^c}+\frac{r}{{v}^c}\right).
\end{align*}
\end{proposition}

Note that in the previous section the four points $\mmd,\ \mmd[2],\ \mmg$ and $\mmg[2]$ are involved whereas in the last proposition these are the points $\mmd,\ \ma,$ $\mmg$ and $\mag$. The former ones are interesting because they simplify computations for Proposition $\ref{majind}$, but, as we shall see in the next section, the latter ones simplify the study of the integrals $\Imd,\ \Ipd,\ \Img$ and $\Ipg$.

\begin{proof}
We prove that on the intersection of all the events of Proposition $\ref{majind}$, the event
$$
\left\{re^{v}\leq \LX^*(\sigp)\leq re^{v}(1+\delta)\ ;\ \isp(1-\delta)\leq\frac{\sigp}{re^{v}}\leq \Isp+2{v}^{6}\delta\right\}
$$
 is realized. As $\sigp$ is the first time the diffusion has "spent a time" $re^{v}$ in one of the points $\mmd,\ \ma,\ \mmg$ or $\mag$ we already have
$re^{v}\leq \LX^*(\sigp)$. Moreover, local time is non-decreasing, so for every $x\in\R$,
$$
\LX(\sigp,x)= \LX(\widehat{\tla}^+_{v},x)\wedge \LX(\smg,x)\wedge \LX(\smd,x)\wedge \LX(\tla^+_{v},x).
$$
On $\Gamma_{v}$, the time $\tla^+_{v}$ is equal to $\smd$ or to $\spd$ and $\widehat{\tla}^+_{v}$ is equal to $\smg$ or to $\spg$. Then the inequality
$\LX^*(\sigp)\leq re^{v}(1+\delta)$ holds.

Continue with the estimate of $\sigp$: by definition of the local time,
$$\sigp=\int_{-\infty}^{+\infty}\LX(\sigp,x)\ud x,\ \Pp\textrm{-a.s.}$$
If $\ma=\mmd$, then $\Cma=\bmd$ and $\tla^+_{v}=\smd$, therefore
\begin{align*}
\int_{0}^{+\infty}\LX(\sigp,&x)\ud x\\
\leq&\int_0^{\amd}\LX(\smd,x)\ud x+\int_{\amd}^{\bmd}\LX(\smd,x)\ud x+\int_{\Cma}^{\Bpd}\LX(\tla^+_{v},x)\ud x\\
\leq&(\Imd+\delta\Bpd)re^{v}
\end{align*}
else if $\ma=\mpd$,
\begin{align*}
\int_{0}^{+\infty}\LX(\sigp,x)\ud x&\leq\int_0^{\amd}\LX(\smd,x)\ud x+\int_{\amd}^{\bmd}\LX(\smd,x)\ud x\\
+\int_{\bmd}^{\apd}&\LX(\spd,x)\ud x+\int_{\apd}^{\bpd}\LX(\spd,x)\ud x+\int_{\Cma}^{\Bpd}\LX(\tla^+_{v},x)\ud x\\
&\leq(\Imd+\Ipd+\delta\Bpd)re^{v}.
\end{align*}
The integral $\int_{-\infty}^{0}\LX(\sigp,x)\ud x$ has a similar upper bound. On $\Gamma_{v}$ we have $\Bpd+\Bpg\leq2{v}^6$ and the upper bound of the proposition follows immediately.

If $\sigp=\smd$, we have
\begin{align*}
\smd\geq\int_{\amd}^{\bmd}\LX(\smd,x)\ud x\geq\Imd(1-\delta)re^{v}.
\end{align*}

The same computation when $\sigp$ takes one of the three other possible values yields to the lower bound stated in the proposition.

Finally, as 
 $$
 \nu_{\sigp}(\overline{A}_{v})=\int_{-\infty}^\infty\mathbf{1}_{\overline{A}_{v}}\LX(\sigp,x)\ud x,
 $$
 we can obtain the bound of the proposition proceeding in the same way as before.
\end{proof}

As the behavior of $\sigp$ is controlled, same kind of results in deterministic time can be obtained.
\begin{proposition}\label{limLX}
For any $0<\delta\leq1/2$, for ${v}$ large enough, if $W\in\Gamma_{v}$,
\begin{align*}
\PW&\bigg(\frac{e^{{v}}}{\Isp+2{v}^{6}\delta}\leq \LX^*(e^{v})\leq \frac{e^{v}(1+\delta)}{\isp(1-\delta)}\bigg)\\
&\geq1-K\left(\frac{1}{\delta\sqrt{{v}^{c-6}}}\exp\left(-\frac{\delta^2{v}^{c-6}}{K}\right)+\exp\left(-\frac{\delta {v}^{c-6}}{K}\right)+\frac{1}{\delta{v}^c}+\frac{1}{{v}^{c-4}}\right).
\end{align*}
\end{proposition}
\begin{proof}
We use the real number $r$ which appears in all propositions since the beginning. Define
$$\rho({v}):=\frac{1}{\Isp+2{v}^{6}\delta}\textrm{ and }r({v}):=\frac{1}{\isp(1-\delta)}.$$
We write $\shp$ for the time $\sigp$ associated with $\rho$ and $\shm$ for the one associated with $r$. Consider the events 
\begin{align*}
\Omega^\rho:=&\left\{\frac{\shp}{\rho({v})e^{v}}\leq\Isp+2{v}^{6}\delta\right\}=\left\{\shp\leq e^{v}\right\},\\
\Omega^r:=&\left\{\LX^*(\shm)\leq r({v})e^{v}(1+\delta)\ ;\ \isp(1-\delta)\leq\frac{\shm}{r({v})e^{v}}\right\}\\
=&\left\{\LX^*(\shm)\leq r({v})e^{v}(1+\delta)\ ;\ e^{v}\leq\shm\right\}.
\end{align*}

As the maximum of the local time is a non decreasing function, on $\Omega^r$, 
$$\LX^*(e^{v})\leq \LX^*(\shm)\leq r({v})e^{v}(1+\delta)$$
and on $\Omega^\rho$, 
$$\rho({v})e^{v}\leq \LX^*(\shp)\leq \LX^*(e^{v}).$$
Therefore it is enough to find a lower bound for $\Pp(\Omega^r\cap\Omega^\rho)$. According to the previous proposition, we only have to estimate $r$ and $\rho$.
First, on $\Gamma_{v}$, the following inequality hold 
$$\isp\leq\Isp\leq\Bpd+\Bpg\leq2{v}^6.$$
Moreover, $\mmd-\amd\geq {v}^{-2}$ and $$\Wb((\mmd-\log{v})\vee\amd,\mmd)-W(\mmd)\leq2\log{v},$$ thus
\begin{align*}
\Imd&\geq\int_{(\mmd-\log{v})\vee\amd}^{\mmd}e^{-W(x)+W(\mmd)}\ud x\geq\frac{1}{{v}^4}.
\end{align*}
The lower bounds for $\Ipd$, $\Img$ and $\Ipg$ are found in the same way.
Finally,
$$ \frac{K}{{v}^6}\leq \rho({v})\leq r({v})\leq2{v}^4$$
and the estimate of the proposition follows easily.
\end{proof}

Using similar arguments, we can also obtain a result in deterministic time for the occupation measure.
\begin{proposition}\label{occdet}
For any $0<\delta\leq1/2$, for ${v}$ large enough, if $W\in\Gamma_{v}$,
\begin{align*}
\PW&\left(\nu_{e^{v}}(\overline{A}_{v})\leq 8{v}^{10}e^{v}\delta\right)\\
&\geq1-K\left(\frac{1}{\delta\sqrt{{v}^{c-6}}}\exp\left(-\frac{\delta^2{v}^{c-6}}{K}\right)+\exp\left(-\frac{\delta {v}^{c-6}}{K}\right)+\frac{1}{\delta{v}^c}+\frac{1}{{v}^{c-4}}\right).
\end{align*}
\end{proposition}


Fix now $c_0>10$ and recall that $c_1=2c+8$. Proposition \ref{occdet} used with $\delta={v}^{-c_0}/8$ and $c>6+2c_0$ and the upper bound for $\Pp(\Gamma_{v})$ of Proposition \ref{probGamma} yield
\begin{equation} \label{est0}
\Pp\left(\nu_{e^{v}}(\overline{A}_{v})\leq e^{v}/{v}^{c_0-10}\right)\geq1-K\left(\frac{\log{v}}{{v}-c_1\log{v}}\right)^2\ \textrm{for ${v}$ large enough}.
\end{equation}

Proposition \ref{limLX} with $\delta={v}^{-7}$ and $c>20$ and the upper bound for $\Pp(\Gamma_{v})$ of Proposition \ref{probGamma} yield for ${v}$ large enough to
\begin{equation}\label{est1}
\Pp\left(\frac{e^{{v}}}{\Isp+2{v}^{-1}}\leq \LX^*(e^{v})\leq \frac{e^{v}(1+{v}^{-7})}{\isp(1-{v}^{-7})}\right)\geq1-K\left(\frac{\log{v}}{{v}-c_1\log{v}}\right)^2.
\end{equation}
and if we use the event $\Gamma_{v}'$ instead of $\Gamma_{v}$, we obtain
\begin{equation}\label{est2}
 \Pp\left(\frac{e^{{v}}}{\Imd+\Img+2{v}^{-1}}\leq \LX^*(e^{v})\leq \frac{e^{v}(1+{v}^{-7})}{\Imd\wedge\Img(1-{v}^{-7})}\right)\geq1-\frac{K\log{v}}{{v}-c_1\log{v}}.
\end{equation}

\section{Proof of Theorems \ref{thcentral} and \ref{thmloc}}\label{demth}
 
As shown by Proposition \ref{limLX}, the asymptotic behavior of $\LX^*$ has a direct link with the ones of $\isp$ and $\Isp$. Therefore, the proof of Theorem \ref{thcentral} requires to study the integrals $\Imdi$, $\Ipd$, $\Img$ and $\Ipg$.

\subsection{Maximum and minimum speed}

Begin with a lower bound for the maximum speed:
\begin{lemme}\label{limsup}Let $\al=\exp(n)$. $\Pmil$-a.s.,
$$\lim\sup_{n\rightarrow\infty} \frac{\Imd[v_n]\wedge\Img[v_n]}{\log_2{v}_n}
\geq \frac{4}{e^2\pi^2}.
$$
\end{lemme}
\begin{proof}
 First, define for $n$ large enough, the sequence of events
$$
E_n:=\left\{\mmm[n]>\mmm[n-1]\ ;\ \int_{\mmm[n]}^{\bbb[n]}e^{-W(x)+W(\mmm[n])}\ud x
\geq \frac{4\log n}{e^2\pi^2}\right\}.
$$
Denote by $(\mathcal{G}_n)$ the filtration generated by $(W(x),\ 0\leq x\leq \bbb[n])$. The process $W_{n}:=(W(x+\bbb[n-1])-W(\bbb[n-1]),\ x\geq0
)$ is a Brownian motion independent of $\mathcal{G}_{n-1}$. The event $E_n$ can be expressed in term of $W_n$: $E_n=E_{n,1}\cap E_{n,2}$ where
\begin{align*}
E_{n,1}
=&\left\{W_{n}\text{ hits $-\al[n-1]+c_1\log \al[n-1]$ before $\al[n]-\al[n-1]-c_1$}\right\}\text{ and}\\
E_{n,2}=&\left\{\int_{\mmm[n](W_n)}^{\bbb[n](W_n)}e^{-W_n(x)+W_n(\mmm[n])}\ud x\geq \frac{4\log n}{e^2\pi^2}\right\}.
\end{align*}
Therefore, $E_n$ is independent of $\mathcal{G}_{n-1}$ and $\mathcal{G}_n$-measurable. Moreover, thanks to Theorem \ref{thtanaka}, $E_{n,1}$ and $E_{n,2}$ are also independent from each other and  
\begin{align*}
 \Pmil(W_{n}&\text{ hits $-\al[n-1]+c_1\log \al[n-1]$ before $\al[n]-\al[n-1]-c_1$})\\
 &=\frac{\al[n]-\al[n-1]-c_1}{\al[n]-c_1\log\al[n]}\geq(1-e^{-1}-c_1e^{-n})
\end{align*}
and 
\begin{align*}
&\Pmil\bigg(\int_{\mmm[n](W_n)}^{\bbb[n](W_n)}e^{-W_n(x)+W_n(\mmm[n])}\ud x\geq \frac{4\log n}{e^2\pi^2}\bigg)\\
\geq&\Pmil\bigg(\int_{0}^{T_R(2)}e^{-R(x)}\ud x\geq \frac{4\log n}{e^2\pi^2}\bigg)\geq\Pmil\left(T_R(2)\geq \frac{4\log n}{\pi^2}\right).
\end{align*}
Then, according to Item $(\ref{lembo2})$ of Lemma \ref{lemboro},
$$
 \Pmil(E_n)\geq\frac{1-e^{-1}-c_1e^{-n}}{K\sqrt{n}}.
$$
We now define the similar event for $\widehat{W}$:
$$
\widehat{E}_n:=\left\{\mmmg[n]>\mmmg[n-1]\ ;\ \int_{\mmm[n]}^{\bbb[n]}e^{-W(x)+W(\mmm[n])}\ud x
\geq \frac{4\log n}{e^2\pi^2}\right\}.
$$
The events $E_n$, $\widehat{E}_n$ are independent, thus 
$$
\Pmil(E_n\cap\widehat{E}_n)=\Pmil(E_n)\Pmil(\widehat{E}_n)\geq\frac{(1-e^{-1}-c_1e^{-n})^2}{K^2n}.
$$
The second Borel-Cantelli lemma yields the conclusion.
\end{proof}

We are not interested in an upper bound of the  minimum speed because this would lead, except for the value of the constant, to the result obtained by Shi in \cite{Shi}. We now look for almost sure bounds. To this end, we study the successive values $\mmu$ the process $(m_{v},\ {v}\geq2)$ can take. These are precisely defined as follows: define $\gam[0]=0$, $h_0=2$ 
and recursively for any $n\in\N$,
\begin{align*}
\bet[n+1]:=&\inf\{x\geq \gam[n]\ ;\ W(x)-\Wu(\gam,x)=h_{n}\},\\
\mmu[n+1]:=&\inf\{x\geq \gam[n]\ ;\ W(x)=\Wu(\gam,\bet[n+1])\},\\
\gam[n+1]:=&\inf\{x\geq \bet[n+1]\ ;\ W(x)=W(\mmu[n+1])\},\\
\et[n+1]:=&\inf\{x\geq \mmu[n+1]\ ;\ W(x)=W(\mmu[n+1])+2\},\\
M_{n+1}:=&\inf\{x\geq \bet[n+1]\ ;\ W(x)=\Wb(\bet[n+1],\gam[n+1])\},\\
h_{n+1}:=&W(M_{n+1})-W(\mmu[n+1])\textrm{ and}\\
\mathcal{F}_{n+1}:=&\sigma\left(W(x),0\leq x\leq \gam[n+1]\right).
\end{align*}

\begin{figure}[ht]%
\begin{picture}(0,0)%
\includegraphics{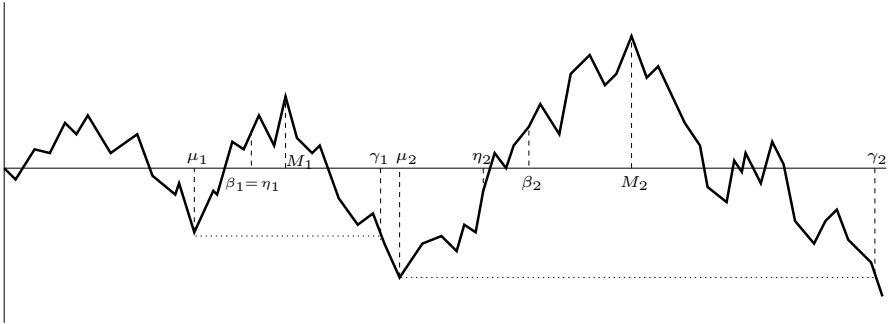}%
\end{picture}%
\setlength{\unitlength}{2072sp}%
\begingroup\makeatletter\ifx\SetFigFont\undefined%
\gdef\SetFigFont#1#2#3#4#5{%
  \reset@font\fontsize{#1}{#2pt}%
  \fontfamily{#3}\fontseries{#4}\fontshape{#5}%
  \selectfont}%
\fi\endgroup%
\begin{picture}(10485,3849)(1903,-5338)
\put(4096,-3391){\makebox(0,0)[lb]{\smash{{\SetFigFont{6}{7.2}{\rmdefault}{\mddefault}{\updefault}{\color[rgb]{0,0,0}$\mmu[1]$}%
}}}}
\put(5266,-3450){\makebox(0,0)[lb]{\smash{{\SetFigFont{6}{7.2}{\rmdefault}{\mddefault}{\updefault}{\color[rgb]{0,0,0}$M_1$}%
}}}}
\put(12151,-3391){\makebox(0,0)[lb]{\smash{{\SetFigFont{6}{7.2}{\rmdefault}{\mddefault}{\updefault}{\color[rgb]{0,0,0}$\gam[2]$}%
}}}}
\put(6256,-3391){\makebox(0,0)[lb]{\smash{{\SetFigFont{6}{7.2}{\rmdefault}{\mddefault}{\updefault}{\color[rgb]{0,0,0}$\gam[1]$}%
}}}}
\put(9226,-3706){\makebox(0,0)[lb]{\smash{{\SetFigFont{6}{7.2}{\rmdefault}{\mddefault}{\updefault}{\color[rgb]{0,0,0}$M_2$}%
}}}}
\put(6571,-3391){\makebox(0,0)[lb]{\smash{{\SetFigFont{6}{7.2}{\rmdefault}{\mddefault}{\updefault}{\color[rgb]{0,0,0}$\mmu[2]$}%
}}}}
\put(4546,-3706){\makebox(0,0)[lb]{\smash{{\SetFigFont{5}{7.2}{\rmdefault}{\mddefault}{\updefault}{\color[rgb]{0,0,0}$\bet[1]\!\!=\!\et[1]$}%
}}}}
\put(8056,-3706){\makebox(0,0)[lb]{\smash{{\SetFigFont{6}{7.2}{\rmdefault}{\mddefault}{\updefault}{\color[rgb]{0,0,0}$\bet[2]$}%
}}}}
\put(7471,-3391){\makebox(0,0)[lb]{\smash{{\SetFigFont{6}{7.2}{\rmdefault}{\mddefault}{\updefault}{\color[rgb]{0,0,0}$\et[2]$}%
}}}}
\end{picture}%

\caption{The variables for a sample path of $W$}
\end{figure}

\begin{lemme}\label{loi integrale}
There is a positive number $K$ such that for any $n>0$ and any $\lambda>0$, 
\begin{align*}
\Pmil\left(\int_{\gam[n-1]}^{M_n}e^{-W(x)+W(\mmu)}\ud x\geq \lambda\right)\leq&Ke^{-j^2_{0}\frac{\lambda}{16}}\textrm{ and}\\
 \Pmil\left(\int_{\mmu}^{\et}e^{-W(x)+W(\mmu)}\ud x\leq \lambda\right)\leq&K\left(2/(e\sqrt{\lambda})+e\sqrt{\lambda}/2\right)e^{-2/(e^2\lambda)}
\end{align*}
where $j_{0}$ is the smallest strictly positive root of the Bessel function $J_0$.
\end{lemme}

\begin{proof}
The process $(W(\gam[n-1]+x)-W(\gam[n-1]),\ x\geq0)$ is a Brownian motion independent of $\mathcal{F}_{n-1}$. Therefore given $h_{n-1}=h$, Theorem $\ref{thtanaka}$ gives the law of the process $$\left(W(\mmu+x)-W(\mmu),\ -\mmu+\gam[n-1]\leq x\leq \bet-\mmu\right).$$
Moreover, according to Proposition 3.13 Chapter 6 of \cite{RevYor}, given $h_{n-1}=h$ and $W(M_n)=M$,
$$
\left(W(x+\bet)-W(\bet)+h_{n-1}),\ 0\leq x\leq M_n-\bet\right)
$$
is a $3$-dimensional Bessel process started at $h$ and killed when it hits $M+h$, thus $\left(W(\mmu+x)-W(\mmu),\ 0\leq x\leq M_n-\mmu\right)$.
is a $3$-dimensional Bessel process started at $0$ and killed when it hits $M+h$.
So if we denote by $R$ and $\widetilde{R}$ two independent Bessel processes of dimension $3$ started at $0$, then
\begin{align*}
\Pmil\left(\int_{\gam[n-1]}^{M_n}e^{-W(x)+W(\mmu)}\ud x\geq \lambda\right)\leq&\Pmil\left(\int_{0}^{\infty}e^{-R(x)}\ud x+\int_{0}^{\infty}e^{-\widetilde{R}(x)}\ud x\geq \lambda\right)\\
\leq&2\Pmil\left(\int_{0}^{\infty}e^{-R(x)}\ud x\geq \frac{\lambda}{2}\right).
\end{align*}
Using Item $(\ref{lembo3})$ of Lemma \ref{lemboro}, 
\begin{align*}
 \Pmil\left(\int_{\gam[n-1]}^{M_n}e^{-W(x)+W(\mmu)}\ud x\geq \lambda\right)\leq&Ke^{-j^2_{0}\frac{\lambda}{16}}.
\end{align*}
For the second bound, we denote
 $T_R(2):=\inf\{x\geq0,R(x)\geq2\}$ and thanks to Theorem $\ref{thtanaka}$ and Item $(\ref{lembo1})$ of Lemma \ref{lemboro}, 
\begin{align*}
 \Pmil\left(\int_{\mmu}^{\et}e^{-W(x)+W(\mmu)}\ud x\leq \lambda\right)=&\Pmil\left(\int_{0}^{T_R(2)}e^{-R(x)}\ud x\leq \lambda\right)\\
\leq&\Pmil\left(T_R(2)\leq e^2\lambda\right)\\
\leq&K\left(2/(e\sqrt{\lambda})+e\sqrt{\lambda}/2\right)e^{-2/(e^2\lambda)}.
\end{align*}
This concludes the proof.
\end{proof}

We also need the following lemma:
\begin{lemme}\label{sautm}
Let $a<\exp(1)<b$. $\Pmil$-a.s., for $n$ large enough,
$$
 a^n<h_n<b^n,\ a^n<W(\mmu)-W(\mmu[n+1])<b^n\ \textrm{and}\ a^{2n}<\gam<b^{2n}.
$$
\end{lemme}
\begin{proof}
Begin with the law of the sequence $(h_n)$. For any $h\geq1$, any $n\in\N$ and any $x\geq 2$,
\begin{align*}
\Pmil\left(\frac{h_{n+1}}{h_n}\leq h|h_n=x\right)=&\Pmil\left(\frac{h_{n+1}-h_n}{h_n}\leq h-1|h_n=x\right)\\
=&\Pmil\left(\tau_W((h-1)x)\geq\tau_W(-x)\right)=1-\frac{1}{h}.
\end{align*}
Thus the variables $r_n:=h_{n+1}/h_n$ are independent and $\log r_n$ is exponentially distributed  with mean $1$. Therefore, $\log h_n-\log h_0=\sum\log r_k$ has the gamma distribution $\Gamma(n,1)$: for any $1<a<\exp(1)$ and $n$ large enough,
\begin{align*}
 \Pmil\left(h_n\leq a^n\right)
 \leq\int_0^{n\log a}\frac{x^{n-1}}{(n-1)!}e^{-x}\ud x\leq\frac{(n\log a)^n}{a^n(n-1)!}
\end{align*}
as the function $x\rightarrow x^{n-1}e^{-x}$ is non decreasing on $[0,n-1]$ and so on $[0,n\log a]$ if $n$ is larger than $(1-\log a)^{-1}$. The
Stirling Formula $n!\sim(\frac{n}{e})^n\sqrt{2\pi n}$ give 
$$
\frac{(n\log a)^n}{a^n(n-1)!}\sim\sqrt{\frac{n}{2\pi}}\left(\frac{e\log a}{a}\right)^n.
$$
As for any $a\in]1,e[$, $0<\frac{e\log a}{a}<1$, the series $\sum\Pmil\left(h_n\leq a^n\right)$ converges. Then the first lower bound is a direct consequence of the Borel-Cantelli lemma. The upper bound is proved in the same way.

For the second result, note that, given $h_{n-1}=x$, $W(\mmu[n-1])-W(\mmu)$ has the same law as $-W(m_{x})$. Therefore,
$$
\Pmil\left(W(\mmu[n-1])-W(\mmu)< h|h_{n-1}=x\right)=1-e^{-h/x}\leq \frac{h}{x}.
$$
Take $1<d<a<\exp(1)$,
\begin{align*}
\Pmil\big(W(&\mmu[n-1])-W(\mmu)< d^{n-1}\big)\\
\leq&\Pmil\left(W(\mmu[n-1])-W(\mmu)< d^{n-1}\ ;\ h_{n-1}> a^{n-1}\right)+\Pmil\left(h_{n-1}\leq a^{n-1}\right)\\
\leq&\left(\frac{d}{a}\right)^{n-1}+\Pmil\left(h_{n-1}\leq a^{n-1}\right).
\end{align*}
The previous proof implies that the sum of $$\Pmil\left(W(\mmu[n-1])-W(\mmu)< d^{n-1}\right)$$ converges and the Borel-Cantelli lemma shows that, almost surely, for large $n$, $$W(\mmu[n-1])-W(\mmu)\geq d^{n-1}.$$  The other bound can be obtained in the same way.

The last inequality with $\gam$ uses same kind of arguments: 
as before, we can show that for any $d>0$,
\begin{align*}
 \Pmil\left(\frac{\gam-\bet}{h_n^2}>d\right)\leq\Pmil\left(\frac{\gam-\bet}{W(\mmu)^2}>d\right)=\Pmil\left(\tau_W(1)\geq d\right)
\end{align*}
and
\begin{align*}
 \Pmil\left(\frac{\bet-\gam[n-1]}{h_n^2}>d\right)\leq\Pmil\left(\frac{\bet-\gam[n-1]}{h_{n-1}^2}>d\right)\leq\Pmil\left(\tau_W(1)\geq d\right).
\end{align*}
Then, 
$$
\Pmil\left(\frac{\gam-\gam[n-1]}{h_n^2}>d\right)\leq2\Pmil\left(\tau_W(1)\geq d/2\right)\leq\frac{K}{\sqrt{d}}.
$$
 Now, let $\epsilon>0$ and $\gam[0]=0$, we have for any $n\geq1$,
\begin{align*}
 \Pmil\left(\frac{\gam}{h_n^2}>(1+\epsilon)^{2n}\right)&\leq\Pmil\left(\sum_{k=1}^{n}\frac{\gam[k]-\gam[k-1]}{h_k^2}>(1+\epsilon)^{2n}\right)\\
\leq\sum_{k=1}^{n}\Pmil&\left(\frac{\gam[k]-\gam[k-1]}{h_k^2}>\frac{(1+\epsilon)^{2n}}{n}\right)\leq\frac{Kn^{3/2}}{(1+\epsilon)^n}.
\end{align*}
Therefore Borel-Cantelli lemma implies that $\Pmil$-a.s., for $n$ large enough,
$
\gam\leq(1+\epsilon)^{2n}h_n^2.
$
 It is then easy to deduce the upper bound for $\gam$. And the lower bound can be obtain easily using same techniques.
\end{proof}

\begin{proposition}\label{liminf}
$\Pmil$-almost surely, 
\begin{align*}
\liminf_{{v}\rightarrow\infty}(\Imd\wedge\Ipd)\log_2{v}\geq&2/e^2\textrm{ and}\\
\limsup_{{v}\rightarrow\infty}\frac{\Imd+\Ipd}{\log_2{v}}\leq&\frac{32}{j^2_{0}}.
\end{align*}
\end{proposition}
\begin{proof} 
Let $k_1\geq0$, $k_2\in\R$ and consider the following integral :
$$
I_v^k:=\int_{a_v^{k_2}}^{b_v^{k_1}}e^{-W(x)+W(m_v)}\ud x
$$
where
\begin{align*}
 b_v^{k_1}:=&\inf\{x\geq m_v,\ W(x)-W(m_v)\geq v-k_1\log v\}\text{ and}\\
 a_v^{k_2}:=&\sup\{x\leq m_v,\ W(x)-W(m_v)\geq v-k_2\log v\}\vee0.
\end{align*}
and $m_v$ is defined in \eqref{defmv}.
The asymptotic behaviour of the integral does not depend on the value of $k_1$ and $k_2$ :
\begin{align}
\liminf_{{v}\rightarrow\infty}I_v^k\log_2{v}\geq&2/e^2\textrm{ and}\label{alfa}\\
\limsup_{{v}\rightarrow\infty}I_v^k(\log_2{v})^{-1}\leq&16/j^2_{0}.\label{romeo}
\end{align}

Begin with the proof of \eqref{alfa}. Fix $d>e^2/2$. According to Lemma \ref{loi integrale}, for any $n\in\N^*$,
$$
\Pmil\left(\int_{\mmu}^{\et}e^{-W(x)+W(\mmu)}\ud x\leq \frac{1}{
d\log(n-1)}\right)\leq K\frac{\sqrt{\log(n-1)}}{(n-1)^{2d/e^2}}.
$$
So the first Borel-Cantelli lemma implies that $\Pmil$-a.s. for $n$ large enough, 
$$\int_{\mmu}^{\et}e^{-W(x)+W(\mmu)}\ud x> \frac{1}{d\log(n-1)}.$$
For $v$ large enough, $\Pmil$-a.s. there is a unique $n\in\N^*$ such that $h_{n-1}<{v}\leq h_{n}$ and $v-k_1\log v>2$. Hence $m_v=\mmu$, $b_v^{k_1}\geq \et$ and
$$
\int_{m_v}^{b_v^{k_1}}e^{-W(x)+W(m_v)}\ud x\geq\int_{\mmu}^{\et}e^{-W(x)+W(\mmu)}\ud x> \frac{1}{d\log(n-1)}.
$$
According to Lemma \ref{sautm}, if $n$ is large enough, $h_{n-1}>2^{n-1}$. Thereby $\Pmil$-a.s., for ${v}$ large enough,
$$
\int_{m_v}^{b_v^{k_1}}e^{-W(x)+W(m_v)}\ud x\geq \frac{1}{d(\log_2{v}-\log_22)}.
$$
When $d$ tends to $e^2/2$, we obtain \eqref{alfa}. 

Continue with the proof of \eqref{romeo}, fix $d>16/j^2_{0}$. One more time, thanks to Lemma \ref{loi integrale} and Borel-Cantelli lemma, $\Pmil$-a.s. for $n$ large enough, 
$$\int_{\gam[n-1]}^{M_n}e^{-W(x)+W(\mmu)}\ud x< d\log (n-1).
$$
For $v$ large enough, $\Pmil$-a.s., there is a unique $n\in\N^*$ such that $h_{n-1}<{v}\leq h_{n}$  and $v-k_1\vee k_2\log v>0$. Therefore $m_v=\mmu$ and $b_v^{k_1}\leq M_n$ and so
\begin{align*}
\int_{a_v^{k_2}}^{b_v^{k_1}}e^{-W(x)+W(m_v)}\ud x\leq&\int_0^{M_n}e^{-W(x)+W(\mmu)}\ud x\\
\leq&\gam[n-1]e^{-W(\mmu[n-1])+W(\mmu)}+\int_{\gam[n-1]}^{M_n}e^{-W(x)+W(\mmu)}\ud x\\
\end{align*}
As $2<\exp(1)<3$, if $n$ is large enough, according to Lemma \ref{sautm}, 
$$\gam[n-1]e^{-W(\mmu[n-1])+W(\mmu)}\leq 3^{2n}e^{-2^n} \text{ and}$$
\begin{align*}
\int_{\gam[n-1]}^{M_n}e^{-W(x)+W(\mmu)}\ud x\leq d\log(n-1)\leq d(\log_2{v}-\log_22).
\end{align*}
Thus,
$$\limsup\frac{1}{\log_2{v}}\int_{a_v^{k_2}}^{b_v^{k_1}}e^{-W(x)+W(m_v)}\ud x\leq d.$$
When $d$ tends to $16/j^2_{0}$, this gives \eqref{romeo}. Then, with properly chosen values for $k_1$ and $k_2$, we get the result of the proposition.
\end{proof}

We can now come back to the local time process.
\subsection{End of the proof of Theorem \ref{thcentral}}
The previous results allow to know the asymptotic behavior of $\LX^*$. Using \eqref{est1} with ${v_n}=n^{2/3}$ and Borel-Cantelli lemma, we obtain, $\Pp$-almost surely for $n$ large enough,
$$
\frac{e^{{v}_n}}{\Isp[{v}_n]+2/v_n}\leq \LX^*(e^{{v}_n})\leq \frac{e^{{v}_n}(1+v_n^{-7})}{\isp[{v}_n](1-v_n^{-7})}.
$$
Thereby Proposition \ref{liminf} gives the following inequalities,
\begin{align*}
&\limsup_{n\rightarrow\infty}\frac{\LX^*(e^{{v}_n})}{e^{{v}_n}\log_2{v}_n}\leq\frac{1}{\liminf\isp[{v}_n]\log_2{v}_n}\leq e^2/2\textrm{ and}\\
&\liminf_{n\rightarrow\infty}\frac{\log_2{v}_n}{e^{{v}_n}}\LX^*(e^{{v}_n})\geq\liminf_{n\rightarrow\infty}\frac{\log_2{v}_n}{\Isp[{v}_n]}\geq \frac{j^2_{0}}{64}.
\end{align*}
Denote by $[x]$ the integer part of $x$, as $\LX^*$ is non decreasing, we get
\begin{align*}
&\frac{j^2_{0}}{64}\leq\liminf_{{v}\rightarrow\infty}\frac{\log_2[{v}^{3/2}]^{2/3}}{e^{[{v}^{3/2}]^{2/3}}}\LX^*(e^{[{v}^{3/2}]^{2/3}})\leq\liminf_{{v}\rightarrow\infty}\frac{\log_2{v}}{e^{v}}\LX^*(e^{v})
\end{align*}
and similarly $\displaystyle \limsup_{{v}\rightarrow\infty}\frac{\LX^*(e^{v})}{e^{v}\log_2{v}}\leq e^2/2$.

\vspace{0.2cm}
For the last inequality of Theorem \ref{thcentral}, $(\ref{est2})$ with ${v_n}=e^n$ and Borel-Cantelli lemma imply that, $\Pp$-almost surely for $n$ large enough,
$$
\frac{e^{{v}_n}}{\Imd[{v}_n]+\Img[{v}_n]+2e^{-n}}\leq \LX^*(e^{{v}_n})\leq \frac{e^{{v}_n}(1+e^{-7n})}{\Imd[{v}_n]\wedge\Img[{v}_n](1-e^{-7n})}.
$$
Then Lemma \ref{limsup} yields directly
$$
\liminf_{{v}\rightarrow\infty}\frac{\log_2{v}}{e^{v}}\LX^*(e^{v})\leq \frac{e^2\pi^2}{4}.
$$    
And the proof of the theorem is completed.

\subsection{End of the proof of Theorem \ref{thmloc}}
Fix $c_0>10$.
According to (\ref{est0}) used with ${v_n}=n$ and Borel-Cantelli lemma, $\Pp$-almost surely for $n$ large enough,
$$
\nu_{e^{n}}(\overline{A}_{n})\leq e^{n}/n^{c_0-10}.
$$
As $t\rightarrow\nu_t(A)$ is a nondecreasing function for every Borel set $A$, $\Pp$-almost surely for $t$ large enough,
$$
\nu_{t}(\overline{A}_{[\log t]+1})\leq e\frac{t}{(\log t)^{c_0-10}}
$$
To obtain the theorem, we need to find a bound for the width of $A_{[\log t]+1}$. Therefore, we only have to estimate the behavior of the processes $\dmd[i]$ and $\emd[i]$. Introduce the following sequences
\begin{align*}
\forall n\geq1,\ \delta_{n}&:=\sup\{\mu_{n}\leq x\leq\beta_{n},\ W(x)-W(\mu_{n})\leq nc_0\log4+\log 8\},\\
\epsilon_{n}&:=\inf\{\gamma_{n-1}\leq x\leq\mu_{n},\ W(x)-W(\mu_{n})\leq nc_0\log4+\log 8\}.
\end{align*}
\begin{lemme}
Let $\epsilon>0$. Then $\Pmil$-a.s., for $n$ large enough,
$$
\delta_{n}-\mu_{n}\leq((n-1)\log 2)^{4+\epsilon}\ \textrm{and}\ \mu_{n}-\epsilon_n\leq((n-1)\log 2)^{4+\epsilon}.
$$
\end{lemme}
\begin{proof}
Notice that for any $u>0$, $n\in\N^*$,
\begin{align*}
&\Pmil(\delta_{n}-\mu_{n}> u|h_{n-1}=x)\\
=&\Pmil\left(\Wu(\mu_n+u,\beta_n)-W(\mu_n)< nc_0\log4+\log 8|h_{n-1}=x\right).
\end{align*}
As the law of the environment near $\mu_n$ is the one of a Bessel process $R$ of dimension $3$ started at $0$ (Theorem \ref{thtanaka}), we have
\begin{align*}
\Pmil(\delta_{n}-\mu_{n}> u|h_{n-1}=x)&=\Pmil\left(\min_{u\leq y\leq\tau_R(x)}R(y)< nc_0\log4+\log 8\right)\\
&\leq \Pmil\left(\min_{u\leq y<\infty}R(y)< nc_0\log4+\log 8\right).
\end{align*}
The left member of the inequality does not depend on $x$, so it is also an upper bound for $\Pmil(\delta_{n}-\mu_{n}< u)$. According to Proposition 3.5, Chap VI in \cite{RevYor}, $\min_{u\leq y<\infty}R(y)$ has the same law as the supremum of a Brownian motion $\max_{0\leq y\leq u}B(y)$, therefore with $u=((n-1)\log 2)^{4+\epsilon}$,
$$
\Pmil(\delta_{n}-\mu_{n}>((n-1)\log 2)^{4+\epsilon})\leq \frac{K}{n^{1+\epsilon/2}}.
$$
So Borel-Cantelli lemma gives the first result and the second one can be obtained in a similar way.
\end{proof}

For ${v}$ large enough, there is a unique integer $n\geq1$ such that $h_{n-1}<{v}-c_1\log v\leq h_n$ and so $\mu_n=\mmd$. Thus, according to Lemma \ref{sautm}, for $v$ large enough, $2^{n-1}\leq {v}-c_1\log v\leq 3^n$ therefore  $\log v\leq n\log 4$ and $$\dmd-\mmd\leq\delta_n-\mu_n\leq (\log 2^{n-1})^{4+\epsilon}\leq(\log v)^{4+\epsilon}.$$
We have the same upper bound for $\mmd-\emd$ and thereby Theorem \ref{thmloc} is proven.

\noindent \\ \textbf{ Acknowledgment:} I thank Romain Abraham and Pierre Andreoletti for their helpful remarks and comments. I am also grateful to an anonymous referee who points out the mistakes of the first version.

\bibliographystyle{amsplain} 
 \bibliography{thbiblio} 

\end{document}